\documentclass[12pt, reqno]{amsart}  

\usepackage{amsmath, amssymb, amsthm}   
\usepackage{mathrsfs}                  
\usepackage{enumitem}                  
\usepackage{graphicx}                 
\usepackage{hyperref} 
\usepackage{esint}
\hypersetup{colorlinks=true, citecolor=blue, linkcolor=blue, urlcolor=blue}

\newtheorem{theorem}{Theorem}[section]

\newtheorem{definition}{Definition}[section]

\newtheorem{lemma}[definition]{Lemma}
\newtheorem{proposition}[definition]{Proposition}
\newtheorem{corollary}[definition]{Corollary}
\theoremstyle{remark}
\newtheorem{remark}[definition]{Remark}
\numberwithin{equation}{section}

\newcommand{\Abs}[1]{\left\lvert#1\right\rvert}
\newcommand{\norm}[1]{\lVert#1\rVert}
\newcommand{\Norm}[1]{\left\lVert#1\right\rVert}
\newcommand{\R}{\mathbb{R}}
\newcommand{\cL}{\mathcal{L}}
\newcommand{\lap}{\mbox{$\triangle$}}
\newcommand{\flap}{\mbox{$(-\triangle)^{\frac{s}{2}}$}}

\newcommand{\al}{\alpha}
\newcommand{\be}{\begin{equation}}
\newcommand{\ee}{\end{equation}}
\newcommand{\bee}{\begin{equation*}}
\newcommand{\eee}{\end{equation*}}

\begin{document}

\title[]{$L^p$ isomorphisms for the fractional Laplacian on $\mathbb{R}^n$ and their application}

\author[C.~Li]{Congming Li}
\address{School of Mathematical Sciences\\
Shanghai Jiao Tong University\\
Shanghai, 200240, China}  \email{congming.li@sjtu.edu.cn}

\author[Y.~Ouyang]{Yugao Ouyang\textsuperscript{*}}
\thanks{\textsuperscript{*}Corresponding author.}
\address{School of Mathematical Sciences\\
Shanghai Jiao Tong University\\
Shanghai, 200240, China}
\email{ouyang1929@sjtu.edu.cn}

\author[Z.~Wang]{Zixuan Wang}
\address{School of Mathematical Sciences\\
Shanghai Jiao Tong University\\
Shanghai, 200240, China}
\email{math\underline{\ }wzx@sjtu.edu.cn}

\begin{abstract}
A classical theory of Amrouche, Girault and Giroire (1994) resolves
the Laplace equation on $\R^n$ through an isomorphism between weighted
Sobolev spaces. Inspired by their framework, we develop the corresponding
$L^p$ theory for the fractional Laplacian: we introduce weighted fractional
Sobolev spaces $\Lambda^{s,p}(\R^n)$, gauged at top order by
$\|(-\triangle)^{\frac{s}{2}}u\|_{L^p}$, and prove that
$(-\triangle)^{\frac{s}{2}}:\Lambda^{s,p}(\R^n)/\mathcal P_{[s-n/p]}\to
L^p(\R^n)$ is an isomorphism for all $s\in(0,2)$ and $p\in(1,\infty)$; thus
$(-\triangle)^{\frac{s}{2}}u=f$ is solvable for every $f\in L^p(\R^n)$,
uniquely modulo an explicit finite-dimensional space of polynomials. More generally, extending the scale by duality to negative orders, the
fractional Laplacians of the appropriate orders map its spaces
isomorphically onto one another and compose exactly. A key ingredient is a family
of weighted Hardy and Poincar\'e inequalities adapted to
$\|(-\triangle)^{\frac{s}{2}}u\|_{L^p}$, established here for $1\le s<2$ by
a reduction to the gradient. Since these inequalities hold a priori only on
a dense class, while the density of $C_{\mathrm{c}}^\infty(\R^n)$ in
$\Lambda^{s,p}(\R^n)$ is itself nontrivial, we first prove the isomorphism
on the closure of the Schwartz class, which coincides with
$\Lambda^{s,p}(\R^n)$ since the difference of two solutions is an entire
$s$-harmonic function, hence a polynomial; the density follows as a
by-product. As an application, we obtain existence and uniqueness, with
explicit kernels and compatibility conditions, for the fractional Stokes
system in $\R^n$.
\end{abstract}



\maketitle

\section{Introduction}\label{sec:intro}

The fractional Laplacian $\flap$, $s\in(0,2)$, is the model nonlocal elliptic
operator. As the infinitesimal generator of the rotationally symmetric
$s$-stable L\'evy process \cite{Bertoin96,Valdinoci09}, it is the natural
diffusion operator for processes with long jumps and memory effects,
describing L\'evy flights and anomalous transport \cite{MK2000}. Within
analysis it pervades the study of nonlocal phenomena: phase transitions and
nonlocal minimal surfaces \cite{CRS10}, the obstacle problem
\cite{Silvestre07}, the surface quasi-geostrophic equation \cite{CV10},
peridynamics and conformal geometry, with further applications ranging from
mathematical finance to image processing and mathematical biology; see
\cite{DPV12,BV16} for systematic introductions. For related regularity
results for fractional equations, see \cite{CLWX26, GLO26}. A basic question underlying
many of these applications is the solvability, on the whole space, of the
inhomogeneous equation
\begin{equation}\label{main1}
    \flap u=f\qquad\text{in }\R^n,
\end{equation}
together with a precise description of the behavior of $u$ at infinity:
\begin{center}
\emph{for which data $f\in L^p(\R^n)$, $1<p<\infty$, does \eqref{main1} have a
solution, in which space does it live, and to what extent is it unique?}
\end{center}
On a bounded domain, with a Dirichlet type condition prescribed on its exterior,
this is by now well understood \cite{RS14,Grubb15}; see \cite{Ros16} for a
survey. On $\R^n$ the absence of a boundary forces one to prescribe the
growth or decay of $u$ at infinity, and the nonlocality of $\flap$ makes even
the meaning of \eqref{main1} a matter requiring care, as we discuss next.

For Schwartz functions, $\flap$ admits several equivalent descriptions: the
singular integral
$c_{n,s}\,\mathrm{p.v.}\!\int_{\R^n}(u(x)-u(y))|x-y|^{-n-s}\,dy$, the Fourier
multiplier $|\xi|^s\widehat u(\xi)$, and the Dirichlet-to-Neumann map of the
Caffarelli--Silvestre extension \cite{CS07}. For functions that merely
grow at infinity, however, the pointwise and multiplier forms may fail to
apply. We therefore use the distributional definition,
$\langle\flap u,\phi\rangle=\int u\,\flap\phi$ for $\phi\in C_{\mathrm{c}}^\infty(\R^n)$
(Definition~\ref{def_fracLap}), which is meaningful for every $u$ in the
growth class $\cL_s$ and is the natural notion for the weighted spaces
studied here.

The classical Laplacian is not an isomorphism on the standard Sobolev
spaces over $\R^n$; the right setting, recovering both an isomorphism and a
description of the solution at infinity, is that of \emph{weighted} Sobolev
spaces, in which derivatives are measured against matching powers of
$\rho(x)=(1+|x|^2)^{1/2}$. Two approaches exist. One rests on the weighted
derivative estimates of Nirenberg and Walker \cite{NW73}, yields a Fredholm
theory for elliptic operators on $\R^n$ \cite{McOwen79,LM83}, and was
subsequently extended to noncompact manifolds with prescribed asymptotic
geometry \cite{Lockhart81,LM85,Bartnik86}; this line lies outside our scope.
The approach we follow originates with Kudryavtsev \cite{Ku59} and Hanouzet
\cite{Hanouzet71} and was brought to completion by Amrouche, Girault and
Giroire \cite{AGG94}. Working in the weighted spaces $W_0^{m,p}(\R^n)$,
which carry powers of $\rho$ with logarithmic corrections at the critical
exponents, they established a complete $L^p$ isomorphism theory: for
Laplace's equation the central result is that
\[
  \lap:\;W_0^{2,p}(\R^n)/\mathcal P_{[2-n/p]}\longrightarrow
 L^p(\R^n)
\]
is an isomorphism, with an analogous statement at every order $m$; the same
framework covers the polyharmonic operators, exterior problems \cite{AGG97},
and the stationary Stokes system \cite{AA99}.

Two structural features persist in the fractional setting: solutions are
unique only modulo a finite-dimensional space of polynomials, and logarithmic
weights are required at the critical exponents. For $\flap$, a comparable
whole-space $L^p$ isomorphism theory has been lacking; the aim of this paper
is to establish it for the full range $s\in(0,2)$.

Before stating the main results, we fix the notation. Throughout,
$\rho(x)=(1+|x|^2)^{1/2}$ and $\lg(x)=\ln(2+|x|^2)$; we write
$\mathcal P_k$ for the polynomials of degree $\le k$ (with
$\mathcal P_k=\{0\}$ for $k<0$), $[t]$ for the integer part of $t$, and
$A\lesssim B$ for $A\le CB$ with $C$ independent of the relevant parameters.
The fractional Laplacian is understood in the distributional sense.
\begin{definition}\label{def_fracLap}
    For $u$ in the growth class
    \[
      \cL_s:=\Big\{u\in L^1_{\mathrm{loc}}(\R^n):
      \int_{\R^n}\frac{|u(x)|}{1+|x|^{n+s}}\,dx<\infty\Big\},
    \]
    the fractional Laplacian $\flap u\in\mathscr D'(\R^n)$ is defined by
    \[
      \langle\flap u,\phi\rangle=\int_{\R^n}u\,\flap\phi,
      \qquad\phi\in C_{\mathrm{c}}^\infty(\R^n).
    \]
\end{definition}

We work in the following weighted fractional Sobolev spaces, in which
derivatives are weighted by powers of $\rho$ and, at the critical exponents,
by the logarithmic factor $\lg$, and the top-order term is
$\|\flap u\|_{L^p}$.
\begin{definition}\label{fS}
    Let $s\in(0,2)$, $p\in(1,\infty)$, $m:=[s]$, and set the split (critical)
    index
    \[
      k:=
      \begin{cases}
        s-\dfrac{n}{p}, & s-\dfrac{n}{p}\in\{0,1,\dots,m\},\\[1.0ex]
        -1, & \text{otherwise.}
      \end{cases}
    \]
    Then the weighted fractional Sobolev spaces is 
    \[
    \begin{aligned}
      \Lambda^{s,p}(\R^n):=\Big\{\,u\in L^1_{\mathrm{loc}}(\R^n):\
        & \rho^{\,|\lambda|-s}\,\lg^{-1}\,D^\lambda u\in L^p,\quad 0\le|\lambda|\le k;\\
        & \rho^{\,|\lambda|-s}\,D^\lambda u\in L^p,\quad k+1\le|\lambda|\le m;\\
        & \flap u\in L^p\,\Big\},
    \end{aligned}
    \]
    normed by the sum of the $L^p$-norms of all components, with seminorm
    $|u|_{\Lambda^{s,p}}:=\|\flap u\|_{L^p}$.
\end{definition}
By H\"older's inequality, the weighted conditions above imply $u\in\cL_s$, so
that $\flap u$ is well defined in every case; moreover, $\Lambda^{s,p}(\R^n)$
is a closed subspace of a finite product of weighted $L^p$-spaces, hence a
reflexive Banach space.

Our main results are a pair of isomorphism theorems that resolve
\eqref{main1} across the full range $s\in(0,2)$, in this scale and in its
dual. The following solves $\flap u=f$ for $L^p$ data.
\begin{theorem}\label{main_isom}
    Let $s\in(0,2)$ and $p\in(1,\infty)$. Then
    \[
        \flap:\Lambda^{s,p}(\R^n)/\mathcal P_{[s-n/p]}\to L^p(\R^n)
    \]
    is an isomorphism, and $C_{\mathrm{c}}^\infty(\R^n)$ is dense in $\Lambda^{s,p}(\R^n)$.
\end{theorem}
\begin{remark}\label{rem_solvability}
    Thus \eqref{main1} is solvable for every $f\in L^p$ by a function
    $u\in\Lambda^{s,p}$, unique modulo the finite-dimensional space
    $\mathcal P_{[s-n/p]}$; when $sp<n$ this space is trivial and the
    solution is unique.
\end{remark}

A duality argument transposes this to the negative-order spaces
$\Lambda^{-s,p}:=(\Lambda^{s,p'})'$: writing $T\perp\mathcal P_k$ when $T$
annihilates $\mathcal P_k$, the operator
\[
  \flap:L^p(\R^n)\longrightarrow
  \Lambda^{-s,p}(\R^n)\perp\mathcal P_{[s-n/p']}
\]
is again an isomorphism (Theorem~\ref{dual_isom}). More generally, these
two theorems are members of a family of composition isomorphisms along the
scale $\{\Lambda^{\sigma,p}(\R^n)\}_{-2<\sigma<2}$, under the convention
$\Lambda^{0,p}:=L^p(\R^n)$: for $0<t\le s$ the operator
$(-\triangle)^{\frac{t}{2}}$, defined on $\mathscr S(\R^n)$ and extended by
continuity, is an isomorphism
\begin{equation}\label{intro_scale_isom}
  (-\triangle)^{\frac{t}{2}}:\
  \Lambda^{s,p}(\R^n)/\mathcal P_{[s-n/p]}
  \longrightarrow
  \Lambda^{s-t,p}(\R^n)/\mathcal P_{[s-t-n/p]},
\end{equation}
with the exact factorization
$(-\triangle)^{\frac{s-t}{2}}\circ(-\triangle)^{\frac{t}{2}}=\flap$ on
$\Lambda^{s,p}(\R^n)$; analogous isomorphisms cross the zeroth order and
continue on the negative side, the polynomial quotients being replaced by
orthogonality constraints, and in general neither can be removed. Since
$\Lambda^{1,p}(\R^n)$ coincides with the weighted Sobolev space
$W_0^{1,p}(\R^n)$ of \cite{AGG94}, the fractional scale contains the
classical first-order space. Precise statements are given in
Section~\ref{sec:iso}; these composition isomorphisms also underlie the
second, $L^p$-pressure setting of the fractional Stokes theorem below.

We now describe the ideas behind the proofs. The first is to gauge
regularity by $\flap$ itself. Since the equation is $\flap u=f$, we build
$\Lambda^{s,p}$ around the seminorm $\|\flap u\|_{L^p}$ rather than the
Gagliardo seminorm $[u]_{s,p}$: with this choice, $L^p$ is the exact image
of $\flap$ and the a priori estimate for $u$ is precisely an estimate of the
solution in terms of the data. The price is that the weighted Hardy and
Poincar\'e inequalities supporting the estimate must be proved for
$\|\flap u\|_{L^p}$ directly; for $0<s<1$ this is achieved by a transfer
from the Gagliardo-seminorm inequalities of Frank--Seiringer \cite{FS08}
and, at the critical exponent, of Nguyen--Squassina \cite{NM18}.

The second idea supplies the missing inequalities for $1\le s<2$ by
inserting the gradient as an intermediate step: a weighted Poincar\'e
inequality, proved through a dyadic decomposition into annuli and covering
the critical exponents with logarithmic weights, bounds $u$ by $Du$; the
case of order $s-1$, already available, then applies to $Du$, and
Calder\'on--Zygmund estimates return $\|\flap u\|_{L^p}$
(Section~\ref{sec:Poincare}).

The third idea is to obtain the density of $C_{\mathrm{c}}^\infty(\R^n)$ in
$\Lambda^{s,p}(\R^n)$ as a consequence of the isomorphism rather than as an
input. Indeed, the inequalities above hold a priori only on a dense class,
while the density of $C_{\mathrm{c}}^\infty(\R^n)$, settled only recently even for the
unweighted spaces of \cite{SS15,CS19} by \cite{BCCS22,KS22}, is unknown in
our weighted setting. We break this circularity by carrying out the whole
scheme on the auxiliary space $Y:=\overline{\mathscr S(\R^n)}$ (closure in
$\Lambda^{s,p}$): we prove the isomorphism
$\flap:Y/\mathcal P_{[s-n/p]}\to L^p$ there and then identify
$Y=\Lambda^{s,p}$, since by the surjectivity on $Y$ any $u\in\Lambda^{s,p}$
differs from an element of $Y$ by an entire $s$-harmonic function, hence by
a polynomial \cite{Fall16}, and polynomials of the admissible degrees lie in
$Y$; the density follows as a by-product (Section~\ref{sec:iso}).

As a concrete application we solve the fractional Stokes system
\begin{equation}\label{intro_stokes}
  \flap u+\nabla P=f,\qquad \mathrm{div}\,u=g\qquad\text{in }\R^n\ (n\ge2),
\end{equation}
a model of steady viscous flow with nonlocal dissipation, whose classical
counterpart ($s=2$) was treated on $\R^n$ in weighted Sobolev spaces by
Alliot and Amrouche \cite{AA99}. For the fractional system itself, the only
existence result we are aware of is due to Cobb \cite{Cobb23}, where
\eqref{intro_stokes} arises as the viscous component of a sedimentation
model: in the divergence-free case $g=0$, for data $f\in L^p(\R^n)^n$ with
$sp<n$, and at the endpoint $sp=n$ for $f\in\dot B^0_{p,1}(\R^n)^n$, a
solution is constructed through the explicit representation
$u=(-\triangle)^{-\frac{s}{2}}\,\mathbb{P}f$, with $\mathbb{P}$ the Leray
projection, and shown to be unique within the class $\mathscr{S}'_0$ of
tempered distributions with vanishing low-frequency part. The polynomial
ambiguity is thus excluded by a normalization rather than characterized, and
the supercritical range $sp>n$ is not treated there. Combining the
isomorphism of Theorem~\ref{main_isom} with the weighted Laplace theory of
\cite{AGG94} for the pressure in the first setting, and with the
composition isomorphisms \eqref{intro_scale_isom} and their duals in the
second, we obtain existence and uniqueness for the full range $s\in(0,2)$
and $p\in(1,\infty)$, with inhomogeneous divergence data and an explicit
description of the kernels, in two functional settings according to the
regularity of the data. Set
\[
    V_{s,p}:=\Lambda^{s-1,p}(\R^n)\perp\mathcal P_{[1-s-n/p']},
\]
so that $V_{s,p}=\Lambda^{s-1,p}(\R^n)$ when $s\ge1$.

\begin{theorem}\label{intro_stokes_thm}
Let $n\ge2$, $s\in(0,2)$, and $p\in(1,\infty)$.
\begin{enumerate}
  \item[\textup{(i)}] Suppose $f\in L^p(\R^n)^n$ and $g\in V_{s,p}$.
        Then \eqref{intro_stokes} has a solution
        $(u,P)\in\Lambda^{s,p}(\R^n)^n\times W_0^{1,p}(\R^n)$, and the
        solutions in this class are exactly the pairs $(u+\lambda,P+c)$
        with $(\lambda,c)$ in
        \[
          \mathcal N_{s,p}
          :=\big\{(\lambda,c)\in(\mathcal P_{[s-n/p]})^n\times
          \mathcal P_{[1-n/p]}:\ \mathrm{div}\,\lambda=0\big\}.
        \]
        Moreover every solution satisfies
        \[
          \inf_{(\lambda,c)\in\mathcal N_{s,p}}
          \big(\|u-\lambda\|_{\Lambda^{s,p}}+\|P-c\|_{W_0^{1,p}}\big)
          \le C\big(\|f\|_{L^p}+\|g\|_{L^p}+\|\flap g\|_{W_0^{-1,p}}\big).
        \]
  \item[\textup{(ii)}] Suppose $f\in\Lambda^{-1,p}(\R^n)^n$ and
        $g\in\Lambda^{s-2,p}(\R^n)$. Then
        \eqref{intro_stokes}---its two equations holding in
        $\Lambda^{-1,p}(\R^n)^n$ and $\Lambda^{s-2,p}(\R^n)$ respectively,
        with the operators extended along the scale as in
        Section~\ref{sec:appstokes}---has a solution
        $(u,P)\in V_{s,p}^n\times L^p(\R^n)$ if and only if
        \[
          f\perp\mathcal P_{[1-n/p']}\quad\text{and}\quad
          g\perp\mathcal P_{[2-s-n/p']}.
        \]
        In that case the pressure $P$ is unique and the solutions are
        exactly the pairs $(u+\lambda,P)$ with
        $\lambda\in(\mathcal P_{[s-1-n/p]})^n$; thus the solution is
        unique whenever $s-1<n/p$, in particular for all $0<s\le1$, and
        unique up to an additive constant vector when $s-1\ge n/p$.
        Moreover every solution satisfies
        \[
          \|P\|_{L^p}
          +\inf_{\lambda\in(\mathcal P_{[s-1-n/p]})^n}
           \|u-\lambda\|_{\Lambda^{s-1,p}}
          \le C\big(\|f\|_{\Lambda^{-1,p}}+\|g\|_{\Lambda^{s-2,p}}\big).
        \]
\end{enumerate}
\end{theorem}

\noindent Part \textup{(i)} is proved as Theorem~\ref{stokes_main} and part
\textup{(ii)} as Theorem~\ref{stokes_Lp_pressure}. The space
$\mathcal N_{s,p}$ is the intersection
$N_s^{\mathrm{St}}\cap\big[\Lambda^{s,p}(\R^n)^n\times W_0^{1,p}(\R^n)\big]$,
where
\[
  N_s^{\mathrm{St}}
  =\big\{(\lambda,c)\in(\mathcal H_s)^n\times\mathcal P_0:\
  \mathrm{div}\,\lambda=0\big\}
\]
is the kernel of the Stokes operator on $(\cL_s)^n\times\mathscr S'(\R^n)$
and $\mathcal H_s:=\{w\in\cL_s:\flap w=0\}$ equals $\mathcal P_0$ for
$0<s\le1$ and $\mathcal P_1$ for $1<s<2$ (Proposition~\ref{kernel}). In each
case the kernel and the orthogonality conditions are explicit, in the spirit
of the local theory of \cite{AGG94,AA99}.

The paper is organized as follows. Section~\ref{sec:Poincare} establishes
the weighted Poincar\'e type inequalities in the subcritical, supercritical,
and critical regimes. Section~\ref{sec:iso} introduces $Y$, transfers these
inequalities to it, proves $\flap:Y/\mathcal P_{[s-n/p]}\to L^p$ is an
isomorphism, and deduces $Y=\Lambda^{s,p}$, from which the density of
$C_{\mathrm{c}}^\infty$ and the full isomorphism follow; it also contains the dual
isomorphism and the mapping properties along the scale.
Section~\ref{sec:appstokes} treats the fractional Stokes system
\eqref{intro_stokes}, first for $L^p$ data with pressure in $W_0^{1,p}$, then
in the $L^p$-pressure setting.

\section{Poincar\'e inequalities}\label{sec:Poincare}

A fundamental observation in \cite{AGG94} is that weighted Poincar\'e inequalities relate the norms of functions to those of their derivatives. In this section, we collect and establish the fractional counterparts.

If $sp<n$, for Schwartz function $u\in {\mathscr S}(\R^n)$, the Riesz potential $I_s(-\triangle)^{s/2} u = u $. Then the classical result of Stein and Weiss\cite{SW58} implies
\begin{equation}\label{sw}
    \int _{\R^n} \frac{|u(x)|^p}{|x|^{sp}}dx\le C\int _{\R^n}|\flap u(x)|^pdx.
\end{equation}

\subsection{The case $0<s<1$}

By \eqref{sw}, in the subcritical range $sp<n$ the Hardy inequality is already
expressed through the seminorm $\norm{\flap u}_{L^p}$. In this subsection we
treat the remaining ranges $sp>n$ and $sp=n$. The known inequalities in these
ranges are phrased in terms of the Gagliardo seminorm
\[
  [u]_{s,p}:=\Big(\int_{\R^n}\int_{\R^n}
  \frac{|u(x)-u(y)|^p}{|x-y|^{n+sp}}\,dx\,dy\Big)^{1/p};
\]
we first record them, and then convert both, in a single statement
(Corollary~\ref{Poincare_0<s<1}), into inequalities governed by
$\norm{\flap u}_{L^p}$, the form used in the rest of the paper.

In the supercritical range, the following Hardy inequality is due to Frank and
Seiringer.

\begin{theorem}[Frank--Seiringer \cite{FS08}]\label{FS}
    Let $n\ge1$, $s\in(0,1)$, and $sp>n$. Then
    \begin{equation}\label{FS_ineq}
        \Norm{\frac{u}{|\cdot|^{s}}}_{L^p(\R^n)}
        \le C_{n,s,p}\,[u]_{s,p},
        \qquad\forall\,u\in C_{\mathrm{c}}^\infty(\R^n\setminus\{0\}).
    \end{equation}
\end{theorem}

In the critical range $sp=n$, a Hardy inequality with logarithmic weight holds
for functions supported outside a ball, due to Nguyen and Squassina; it follows
from \cite[Theorem 3.1]{NM18} with $\tau=p$, $\gamma=-s$, and $\alpha=0$.

\begin{theorem}[Nguyen--Squassina \cite{NM18}]\label{NS}
    Let $s\in(0,1)$, $sp=n$, and $R>0$. Then there exists $C_R>0$ such that
    \begin{equation}\label{ns_ineq}
        \norm{u\,\rho^{-s}\lg^{-1}}_{L^p(\R^n)}
        \le C_R\,[u]_{s,p},
        \qquad\forall\,u\in C_c^1(B_R^c).
    \end{equation}
\end{theorem}

The passage from $[u]_{s,p}$ to $\norm{\flap u}_{L^p}$ rests on the following
comparison.

\begin{lemma}\label{Gagliardo_vs_flap}
    Let $0<s<1$ and $2\le p<\infty$. Then
    \begin{equation}\label{Lambda_to_W}
        [u]_{s,p}\le C_{n,s,p}\,\norm{\flap u}_{L^p(\R^n)},
        \qquad\forall\,u\in C_{\mathrm{c}}^\infty(\R^n).
    \end{equation}
\end{lemma}

\begin{proof}
    For $0<s<2$ and $1<p<\infty$, the norm
    $\norm{u}_{L^p}+\norm{\flap u}_{L^p}$ is equivalent to the norm of the
    Bessel potential space of order $s$ \cite{Stein61}; since $p\ge2$, this
    space embeds into the Besov space $B^s_{p,p}(\R^n)$, whose norm is
    equivalent to $\norm{u}_{L^p}+[u]_{s,p}$ \cite{Stein70}.
    Hence
    \[
        [u]_{s,p}\le C\big(\norm{u}_{L^p}+\norm{\flap u}_{L^p}\big),
        \qquad u\in C_{\mathrm{c}}^\infty(\R^n).
    \]
    Apply this bound to $u_\lambda:=u(\lambda\,\cdot)$; since
    \[
        [u_\lambda]_{s,p}=\lambda^{s-\frac np}[u]_{s,p},\qquad
        \norm{\flap u_\lambda}_{L^p}=\lambda^{s-\frac np}\norm{\flap u}_{L^p},
        \qquad
        \norm{u_\lambda}_{L^p}=\lambda^{-\frac np}\norm{u}_{L^p},
    \]
    dividing by $\lambda^{s-\frac np}$ and letting $\lambda\to\infty$ removes
    the zeroth-order term and yields \eqref{Lambda_to_W}.
\end{proof}

Combining the two theorems above with Lemma~\ref{Gagliardo_vs_flap}, we obtain
the Hardy inequalities in the form used in the sequel.

\begin{corollary}\label{Poincare_0<s<1}
    Let $n\ge2$ and $s\in(0,1)$.
    \begin{itemize}
        \item[\textup{(i)}] If $sp>n$, then
        \begin{equation}\label{Poincare_s<1}
            \Norm{\frac{u}{|\cdot|^{s}}}_{L^p(\R^n)}
            \le C\,\norm{\flap u}_{L^p(\R^n)},
            \qquad\forall\,u\in C_{\mathrm{c}}^\infty(\R^n\setminus\{0\}).
        \end{equation}
        \item[\textup{(ii)}] If $sp=n$ and $R>0$, then
        \begin{equation}\label{critical_Hardy_inequality}
            \norm{u\,\rho^{-s}\lg^{-1}}_{L^p(\R^n)}
            \le C_R\,\norm{\flap u}_{L^p(\R^n)},
            \qquad\forall\,u\in C_{\mathrm{c}}^\infty(B_R^c).
        \end{equation}
    \end{itemize}
\end{corollary}

\subsection{The case $1\le s<2$}

In this subsection, we establish the Poincar\'e inequality for $1\leq s<2\leq n$. For $1<s<2$, the fractional Hardy inequality at order $s-1\in(0,1)$, which
below converts the weighted gradient term into
$\norm{(-\triangle)^{\frac{s-1}{2}}Du}_{L^p}$, is supplied by \eqref{sw} when
$(s-1)p<n$ and by Corollary~\ref{Poincare_0<s<1} when $(s-1)p\geq n$.

The rescaled form of the classical Poincar\'e inequality in a bounded domain is useful:
\begin{lemma}\label{classical_Poincare}
    Let $1<p<\infty$, $A$ be a bounded domain and $E\subset A$ of positive measure. Then
    $$\|u\|_{L^p(A)}\lesssim\|Du\|_{L^p(A)},\quad\forall u\in W^{1,p}(A)\text{ s.t. }\int_Eu=0.$$
\end{lemma}

\begin{proposition}\label{Poincare_dyadic}
    Let $s>0$, $\frac{n}{s}<p<\infty$, $\alpha\geq0$. Then
    $$\int_{\R^n}\frac{|u(x)|^p}{|x|^{sp}\lg^\alpha(x)}dx\lesssim\int_{\R^n}\frac{|Du(x)|^p}{|x|^{(s-1)p}\lg^\alpha(x)}dx,\quad\forall u\in C^1_{\mathrm{c}}(\R^n\setminus B_1).$$
\end{proposition}
\begin{proof}
    We apply the dyadic decomposition of the space. For $k\in\mathbb{N}^*$, we set
    $$A_k=\big\{x\in\R^n:2^{k-1}\leq|x|<2^k\big\},\quad\mu_k=\fint_{A_k}u.$$
    Since $|x|^{sp}\lg^\alpha(x)\sim2^{ksp}k^\alpha$ for $x\in A_k$, we have
    \begin{equation}\label{Poincare_proof}
        \int_{\R^n}\frac{|u(x)|^pdx}{|x|^{sp}\lg^\alpha(x)}\lesssim\sum_{k=1}^\infty\frac{1}{2^{skp}k^\alpha}\int_{A_k}|u|^p\leq\sum_{k=1}^\infty\frac{|A_k|}{2^{skp}k^\alpha}|\mu_k|^p+\sum_{k=1}^\infty\frac{1}{2^{skp}k^\alpha}\int_{A_k}|u-\mu_k|^p.
    \end{equation}
    The classical Poincar\'e inequality yields
    $$\int_{A_k}|u-\mu_k|^p\lesssim 2^{kp}\int_{A_k}|Du|^p;$$
    hence
    $$\sum_{k=1}^\infty\frac{1}{2^{skp}k^\alpha}\int_{A_k}|u-\mu_k|^p\lesssim\sum_{k=1}^\infty\frac{2^{kp}}{2^{skp}k^\alpha}\int_{A_k}|Du|^p\lesssim\int_{\R^n}\frac{|Du(x)|^p}{|x|^{(s-1)p}\lg^\alpha(x)}dx.$$
    To deal with the first summation, we apply Poincar\'e inequality:
    $$|\mu_k-\mu_{k+1}|\leq\Big(\fint_{A_k\cup A_{k+1}}|u-\mu_k|^p\Big)^{\tfrac{1}{p}}+\Big(\fint_{A_k\cup A_{k+1}}|u-\mu_{k+1}|^p\Big)^{\tfrac{1}{p}}\lesssim2^k\Big(\fint_{A_k\cup A_{k+1}}|Du|^p\Big)^{\tfrac{1}{p}},$$
    therefore
    $$|\mu_{k+1}|^p\leq\left(|\mu_k|+2^kC\Big(\fint_{A_k\cup A_{k+1}}|Du|^p\Big)^{\tfrac{1}{p}}\right)^p\leq(1+\epsilon)|\mu_k|^p+C_\epsilon2^{kp}\fint_{A_k\cup A_{k+1}}|Du|^p,$$
    $$\frac{|\mu_{k+1}|^p}{2^{(sp-n)(k+1)}(k+1)^\alpha}\leq\frac{(1+\epsilon)k^\alpha}{2^{sp-n}(k+1)^\alpha}\frac{|\mu_k|^p}{2^{(sp-n)k}k^\alpha}+\frac{C_\epsilon}{2^{(s-1)p(k+1)}(k+1)^\alpha}\int_{A_k\cup A_{k+1}}|Du|^p.$$
    Due to $sp-n>0$, we may choose $\epsilon$ so small that $\frac{(1+\epsilon)k^\alpha}{2^{sp-n}(k+1)^\alpha}\leq1-\delta$ for some $\delta>0$. Taking the sum, we obtain
    $$\sum_{k=1}^{N+1}\frac{|\mu_k|^p}{2^{(sp-n)k}k^\alpha}\leq(1-\delta)\sum_{k=0}^N\frac{|\mu_k|^p}{2^{(sp-n)k}k^\alpha}+C\sum_{k=0}^N\frac{1}{2^{(s-1)pk}k^\alpha}\int_{A_k\cup A_{k+1}}|Du|^p,$$
    $$\sum_{k=1}^\infty\frac{|\mu_k|^p}{2^{(sp-n)k}k^\alpha}\leq C'\sum_{k=1}^\infty\frac{1}{2^{(s-1)pk}(k+1)^\alpha}\int_{A_k}|Du|^p.$$
    Then \eqref{Poincare_proof} implies the assertion.
\end{proof}

\begin{corollary}\label{Poincare_1_s_2_noncritical}
    Applying the above Poincar\'e inequality to $Du$ for $s-1\in(0,1)$ and the Calder\'on--Zygmund theorem, passing to a rescaling if necessary, the following assertions hold:
    \begin{itemize}
        \item If $n<p<\infty$, then for $u\in C_{\mathrm{c}}^\infty(\R^n\setminus\{0\})$,
        $$\Big\|\frac{u}{|\cdot|}\Big\|_{L^p(\R^n)}\lesssim\|Du\|_{L^p(\R^n)}\lesssim\|(-\triangle)^{1/2}\,u\|_{L^p}.$$
        \item If $1<s<2$, $\frac{n}{s}<p<\infty$, $(s-1)p\neq n$, then for $u\in C_{\mathrm{c}}^\infty(\R^n\setminus\{0\})$,
        $$\Big\|\dfrac{u}{|\cdot|^s}\Big\|_{L^p(\R^n)}\lesssim\Big\|\dfrac{Du}{|\cdot|^{s-1}}\Big\|_{L^p(\R^n)}\lesssim\|(-\triangle)^{\tfrac{s-1}{2}}Du\|_{L^p(\R^n)}\lesssim\|\flap u\|_{L^p(\R^n)}.$$
        \item If $1<s<2$, $(s-1)p=n$, then for $u\in C_{\mathrm{c}}^\infty(\R^n\setminus B_1)$,
        $$\Big\|\frac{u}{\rho^s\lg}\Big\|_{L^p(\R^n)}\lesssim\Big\|\frac{Du}{\rho^{s-1}\lg}\Big\|_{L^p(\R^n)}\lesssim\|(-\triangle)^{\tfrac{s-1}{2}}Du\|_{L^p(\R^n)}\lesssim\|\flap u\|_{L^p(\R^n)}.$$
    \end{itemize}
\end{corollary}

Now we derive the inequality for critical values $sp=n$ with $s\ge 1$:
\begin{proposition}\label{Poincare_critical_s_ge_1}
    Let $1\leq s<2$, $1<p<\infty$, $sp=n$. Then
    $$\int_{\R^n}\frac{|u(x)|^p}{|x|^n\lg^p(x)}dx\lesssim\int_{\R^n}\frac{|Du(x)|^p}{|x|^{(s-1)p}}dx,\quad u\in C^1_{\mathrm{c}}(\R^n\setminus B_1).$$
\end{proposition}
\begin{proof}
    With the same notations as in the proof of Proposition~\ref{Poincare_dyadic}, we write
    \begin{equation}
        \int_{\R^n}\frac{|u(x)|^pdx}{|x|^n\lg^p(x)}\lesssim\sum_{k=1}^\infty\frac{1}{2^{kn}k^p}\int_{A_k}|u|^p\leq\sum_{k=1}^\infty\frac{|A_k|}{2^{kn}k^p}|\mu_k|^p+\sum_{k=1}^\infty\frac{1}{2^{kn}k^p}\int_{A_k}|u-\mu_k|^p,
    \end{equation}
    where the last term is dominated by
    $$\sum_{k=1}^\infty\frac{2^{kp}}{2^{nk}k^p}\int_{A_k}|Du|^p\lesssim\int_{\R^n}\frac{|Du(x)|^p}{|x|^{(s-1)p}}dx.$$
    We apply Lemma~\ref{precise_pointwise_interpolation} below \cite[Lemma 3.2]{NM18} with $c_k=\frac{(k+1)^{p-1}}{(k+1/2)^{p-1}}$ to
    $$|\mu_{k+1}|^p\leq\left(|\mu_k|+2^k\Big(\fint_{A_k\cup A_{k+1}}|Du|^p\Big)^{\tfrac{1}{p}}\right)^p$$
    to obtain that
    $$\frac{|\mu_{k+1}|^p}{(k+1)^{p-1}}\leq\frac{c_k|\mu_k|^p}{(k+1)^{p-1}}+\frac{2^{kp}C}{(k+1)^{p-1}(c_k-1)^{p-1}}\fint_{A_k\cup A_{k+1}}|Du|^p.$$
    The sequence $\{(k+1)(c_k-1)\}$ has a positive lower bound. Hence
    $$\sum_{k=1}^{N+1}\frac{|\mu_k|^p}{k^{p-1}}\leq\sum_{k=0}^N\frac{|\mu_k|^p}{(k+1/2)^p}+C\sum_{k=0}^N2^{(p-n)k}\int_{A_k\cup A_{k+1}}|Du|^p.$$
    Subtracting the first term on the right-hand side to the left, by $k^{-p}\lesssim k^{1-p}-(k+1/2)^{1-p}$ we see
    $$\sum_{k=1}^N\frac{|\mu_k|^p}{k^p}\lesssim\sum_{k=1}^\infty2^{(p-n)k}\int_{A_k}|Du|^p\lesssim\int_{\R^n}\frac{|Du(x)|^p}{|x|^{(s-1)p}}dx.$$
\end{proof}

\begin{lemma}[\cite{NM18}]\label{precise_pointwise_interpolation}
    Let $\Lambda>1$ and $\tau>1$. There exists $C=C(\Lambda,\tau)>0$, depending only on $\Lambda$ and $\tau$, such that for all $1<c<\Lambda$,
    $$(a+b)^\tau\leq ca^\tau+\frac{C}{(c-1)^{\tau-1}}b^\tau.$$
\end{lemma}

\begin{corollary}\label{Poincare_sp_n}
    Applying Hardy's inequality to $Du$ for $s-1\in(0,1)$ and the Calder\'on--Zygmund theorem, the following holds for $u\in C_{\mathrm{c}}^\infty(\R^n\setminus B_1)$:
    \begin{itemize}
        \item $\Big\|\dfrac{u}{\rho\lg}\Big\|_{L^n(\R^n)}\lesssim\|Du\|_{L^p(\R^n)}\lesssim\|(-\triangle)^{1/2}\,u\|_{L^p(\R^n)}$.
        \item If $1<s<2$, $1<p<\infty$, $sp=n$, then
        $$\Big\|\frac{u}{\rho^s\lg}\Big\|_{L^p(\R^n)}\lesssim\Big\|\frac{Du}{\rho^{s-1}}\Big\|_{L^p(\R^n)}\lesssim\|\flap u\|_{L^p(\R^n)}.$$
    \end{itemize}
\end{corollary}

\section{Isomorphism theorems}\label{sec:iso}

In this section, we prove the main isomorphism theorem (Theorem~\ref{main_isom}) and some other related isomorphisms. The strategy to prove Theorem~\ref{main_isom} is as follows: we first define $Y$ as the closure of test functions in $\Lambda^{s,p}(\R^n)$, prove that $\flap :Y/\mathcal{P}_{[s-n/p]}\to L^p(\R^n)$ is an isomorphism, and then show $Y=\Lambda^{s,p}(\R^n)$.

\subsection{The Poincar\'e inequality on $Y$}

One can easily checked that Schwartz space ${\mathscr S}(\R^n)\in \Lambda ^{s,p}(\R^n)$, thus the following definition makes sense:
\begin{definition}\label{def_Y}
    Let $s\in(0,2)$ and $p\in(1,\infty)$. Define $Y$ to be the closure of
    $\mathscr S(\R^n)$ in $\Lambda^{s,p}(\R^n)$.
\end{definition}

We first state an important remainder estimate, which is derived by the
fractional Leibniz rule of Li \cite{Li19}
(with the case $s\in(0,1)$ originating in the commutator estimates of
Kenig, Ponce and Vega \cite{KPV93}). For $s\in[1,2)$, a first-order
term is present, which we then absorb.

\begin{lemma}\label{remainder_estimate}
    Let $s\in(0,2)$ and $p\in(1,\infty)$. Then for all $f, g\in\mathscr S(\R^n)$, define the remainder term
   \[
   R_s(f,g) := \flap(fg)-f\,\flap g-g\,\flap f.
   \]
   Then
   \[
       \norm{R_s(f,g)}_{L^p(\R^n)}\lesssim \begin{cases}
           \norm{\flap f}_{L^p(\R^n)}\norm {g}_{L^\infty(\R^n)}, &\text{if }s\in (0,1),\\
          \norm{(-\triangle)^{\frac{s-1}{2}}f}_{L^p(\R^n)}(\norm {Dg}_{L^\infty(\R^n)}+\norm {(-\triangle)^{1/2}g}_{L^\infty(\R^n)}), &\text{if }s\in [1,2).
       \end{cases}
   \]
\end{lemma}
\begin{proof}
    \cite[Case 1 (2) in Theorem 1.2]{Li19} states that for any $s_1,s_2\ge 0$ with $s_1+s_2=s$,
    \begin{equation}
    \begin{split}
         \Norm{\flap (fg)-\sum _{|\alpha|<s_1}\frac{1}{\alpha !}\partial ^\alpha fD^{s,\alpha}g-\sum _{|\beta|\le s_2}\frac{1}{\beta!}\partial ^{\beta}gD^{s,\beta}f}_{L^p(\R^n)}\\
         \lesssim _{s,s_1,s_2,p,n}\norm{(-\triangle)^{s_1/2}f}_{L^p(\R^n)}\norm{(-\triangle)^{s_2/2}g}_{BMO},
    \end{split}
    \end{equation}
    where $\widehat{D^{s,\alpha}g}(\xi)= i^{-|\alpha|}\partial _{\xi}^\alpha(|\xi|^s)\widehat{g}(\xi)$. Noticing that $D^{s,0}=\flap$ and $D^{s,e_j}= sR_j(-\triangle)^{\frac{s-1}{2}}$, where $R_j$ denotes the Riesz
transform normalized by $\widehat{R_jh}(\xi)=i\,\xi_j|\xi|^{-1}\hat h(\xi)$
and is bounded on $L^p(\R^n)$.

Now the lemma is derived by taking $s_1=s,s_2=0$ when $s\in (0,1)$; $s_1=s-1, s_2=1$ when $s\in [1,2)$.
\end{proof}

We shall also use the following fractional Gagliardo--Nirenberg
inequalities, which are special cases of the weighted results of
\cite{DDS23}:

\begin{lemma}\label{moment_GN}
    Let $p\in(1,\infty)$ and $0<\alpha<\beta\le2$, with the convention
    $(-\triangle)^{\frac{2}{2}}:=-\triangle$. Then
    \begin{equation}\label{moment_ineq}
        \norm{(-\triangle)^{\frac{\alpha}{2}}f}_{L^p(\R^n)}
        \lesssim_{\alpha,\beta,p}
        \norm{f}_{L^p(\R^n)}^{1-\frac{\alpha}{\beta}}\,
        \norm{(-\triangle)^{\frac{\beta}{2}}f}_{L^p(\R^n)}^{\frac{\alpha}{\beta}},
        \qquad f\in\mathscr S(\R^n).
    \end{equation}
    In particular, for all $f\in\mathscr S(\R^n)$,
    \begin{align}
        \norm{(-\triangle)^{\frac{\sigma}{2}}f}_{L^p(\R^n)}
        &\lesssim\norm{f}_{L^p(\R^n)}^{1-\sigma}\,
          \norm{Df}_{L^p(\R^n)}^{\sigma},
        &&0<\sigma<1,\label{GN_first}\\
        \norm{\flap f}_{L^p(\R^n)}
        &\lesssim\norm{f}_{L^p(\R^n)}^{1-\frac{s}{2}}\,
          \norm{D^2f}_{L^p(\R^n)}^{\frac{s}{2}},
        &&0<s<2,\label{GN_second}\\
        \norm{Df}_{L^p(\R^n)}
        &\lesssim\norm{f}_{L^p(\R^n)}^{1-\frac{1}{s}}\,
          \norm{\flap f}_{L^p(\R^n)}^{\frac{1}{s}},
        &&1<s<2.\label{GN_grad}
    \end{align}
\end{lemma}

\begin{proof}
    Inequality \eqref{moment_ineq} is the unweighted, equal-exponent
    case of \cite[Theorem 1.6]{DDS23}.
\end{proof}

Then we have the density of $C_{\mathrm{c}}^\infty(\R^n)$:

\begin{lemma}\label{C_c_dense_in_Y_s_ge_1}
 $C_{\mathrm{c}}^\infty(\R^n)$ is dense in $Y$.
\end{lemma}
\begin{proof}
    Let $u\in\mathscr{S}(\R^n)$. Choose a nonnegative cutoff $\eta\in C_{\mathrm{c}}^\infty(B_2)$ with $\eta=1$ on $B_1$ and set $\eta_R(x)=\eta(x/R)$. Denote $v_R=(1-\eta _R)u$. It is clear that $\|\rho^{-s}v_R\|_{L^p}$ and  $\|\rho^{1-s}Dv_R\|_{L^p}$ tend to 0 as $R\to 0$ by the dominated convergence theorem. For the highest order term, Lemma~\ref{moment_GN} implies:
    \[
    \|\flap v_R\|_{L^p}\lesssim \|v_R\|_{{L^p}}^{1-\frac{s}{2}}\|D^2v_R\|_{L^p}^{\frac{s}{2}}\to 0\text{ as }R\to\infty
    \]
    since $\|v_R\|_{{L^p}}\to 0$ as $R\to\infty$ and $\|D^2v_R\|_{L^p}$ are uniformly bounded. 
\end{proof}

 To transfer the Poincar\'e inequalities of Section~\ref{sec:Poincare} from $C_{\mathrm{c}}^\infty(\R^n\setminus\{0\})$ to $Y$, we need the following density result.

\begin{proposition}\label{density_away_from_0}
    Let $sp\ge n$ and $s-\frac{n}{p} \notin\{0,1\}$. The point evaluation $Y\ni u\mapsto D^\al u(0)$ is bounded for $|\al|<s-\frac{n}{p}$. Moreover, $C^\infty_\mathrm{c}(\R^n\setminus\{0\})$ is dense in
    \[
    Y_0:=\big\{u\in Y:D^\alpha u(0)=0,\;|\al|<s-\tfrac{n}{p}\big\}.
    \]
\end{proposition}
\begin{proof}
    The boundedness of point evaluation follows from Morrey's embedding.

    \medskip\noindent\textit{Step 1.} We show $\{u\in C_{\mathrm{c}}^\infty(\R^n): D^\alpha u(0)=0,\;|\alpha|<s-n/p\}\hookrightarrow Y_0$ is dense. Let $\eta\in C_{\mathrm{c}}^\infty(B_2)$ with $\eta|_{B_1}=1$. For $u\in Y_0$, by Lemma~\ref{C_c_dense_in_Y_s_ge_1}, choose $\{u_k\}\subset C_{\mathrm{c}}^\infty(\R^n)$ with $\norm{u_k-u}_{\Lambda^{s,p}}\to0$. Morrey's embedding gives
    $$|D^\alpha(u_k-u)(0)|\lesssim\norm{u_k-u}_{\Lambda^{s,p}}\to0,\quad|\alpha|<s-\tfrac{n}{p},$$
    hence $\|P_{u_k,0}\eta\|_{\Lambda^{s,p}}\to0$, where $P_{u_k,0}(x):=\sum_{|\alpha|\leq s-n/p}\frac{D^\alpha u_k(0)}{\alpha!}x^\alpha$. Therefore $u_k-P_{u_k,0}\eta$ converges to $u$ in $\Lambda^{s,p}(\R^n)$ and satisfies $D^\alpha(\cdot)(0)=0$ for $|\alpha|<s-n/p$.

    \medskip\noindent\textit{Step 2.} We approximate any $u\in C_{\mathrm{c}}^\infty(\R^n)$ with $D^\alpha u(0)=0$ for $|\alpha|< s-n/p$ by $C_{\mathrm{c}}^\infty(\R^n\setminus\{0\})$ functions. Set $\eta_\epsilon(x)=\eta(x/\epsilon)$. We claim $\norm{\eta_\epsilon u}_{\Lambda^{s,p}}\to0$, which immediately implies $(1-\eta_\epsilon)u\to u$. The dominated convergence theorem yields
    $$\norm{\eta_\epsilon u\rho^{-s}}_{L^p}+\norm{D(\eta_\epsilon u)\rho^{1-s}}_{L^p}\to0.$$
    Using $\norm{u}_{L^\infty(B_\epsilon)}\lesssim\epsilon^{1+[s-n/p]}$ and $\norm{Du}_{L^\infty(B_\epsilon)}\lesssim\epsilon^{[s-n/p]}$, one derives
    \[
        \norm{\eta_\epsilon u}_{L^p}\lesssim\epsilon^{1+[s-n/p]+n/p},
    \]
    and
    \begin{align*}
        \norm{D^2(\eta_\epsilon u)}_{L^p}&\leq\|uD^2\eta_\epsilon\|_{L^p(B_\epsilon)}+2\|DuD\eta_\epsilon\|_{L^p}+\|\eta_\epsilon D^2u\|_{L^p}\lesssim\epsilon^{[s-n/p]+n/p-1}.
    \end{align*}
    Then \eqref{GN_second} implies $\norm{\flap(\eta_\epsilon u)}_{L^p}\to0$.
\end{proof}

Now we are ready to prove the following Poincar\'e inequality on $Y$. 

\begin{proposition}\label{Poincare_on_Y}
    Let $s\in(0,2)$, $sp\ge n$. Then
    $$\inf_{q\in\mathcal{P}_{[s-n/p]}}\|u-q\|_{\Lambda^{s,p}(\R^n)}\lesssim\|\flap u\|_{L^p(\R^n)},\quad u\in Y.$$
\end{proposition}
\begin{proof}
    We divide the proof according to whether $s-n/p$ is a critical value.

    \medskip\noindent\textbf{Non-critical case}: $s-\frac{n}{p}\notin\{0,1\}$. For $u\in Y$, let $P_{u,0}$ be the $[s-n/p]$-th Taylor expansion of $u$ at $0$, so that $D^\alpha(u-P_{u,0})(0)=0$ for $|\alpha|<s-n/p$, i.e. $u-P_{u,0}\in Y_0$. By Proposition~\ref{density_away_from_0}, $C_{\mathrm{c}}^\infty(\R^n\setminus\{0\})$ is dense in $Y_0$. Combining with Corollary~\ref{Poincare_0<s<1} (for $s<1$) or Corollary~\ref{Poincare_1_s_2_noncritical} (for $s\ge 1$), the Poincar\'e inequality extends from $C_{\mathrm{c}}^\infty(\R^n\setminus\{0\})$ to $Y_0$:
    $$\norm{u-P_{u,0}}_{\Lambda^{s,p}(\R^n)}\lesssim\|\flap u\|_{L^p(\R^n)}.$$

    \medskip\noindent\textbf{Critical case}: $s-\frac{n}{p}\in\{0,1\}$.
    For $u\in Y$, we choose the unique $U\in u+\mathcal{P}_{[s-n/p]}$ such that:
    \begin{equation}\label{zero_average}
    \begin{cases}
        \displaystyle\int_{B_1}U=0,&\text{if }~s-\dfrac{n}{p}=0;\\
        \displaystyle U(0)=0,\;\int_{B_1}DU=0,&\text{if }~s-\frac{n}{p}=1.
    \end{cases}
    \end{equation}
    We claim $\|U\|_{\Lambda^{s,p}}\lesssim\|\flap U\|_{L^p}$. Otherwise, there exists $\{U_k\}\subset Y$ satisfying \eqref{zero_average} with $\norm{U_k}_{\Lambda^{s,p}}=1$ and $\norm{\flap U_k}_{L^p}\to0$. Since $\Lambda^{s,p}$ is reflexive, up to a subsequence, $U_k\rightharpoonup U_*$ in $\Lambda^{s,p}$. Then $\|\flap U_*\|_{L^p}\leq\liminf_{k\to\infty}\|\flap U_k\|_{L^p}=0$, so $U_*\in\mathcal{P}_{[s-n/p]}$. By weak convergence, $U_*$ satisfies \eqref{zero_average}, hence $U_*=0$.

    Let $\phi,\psi$ be a partition of unity of $\R^n$ subordinate to $\{B_2,\,\R^n\setminus\bar B_1\}$. We shall prove that $\|\phi U_k\|_{\Lambda^{s,p}}$ and $\|\psi U_k\|_{\Lambda^{s,p}}$ both tends to 0 as $k\to\infty$, which contradicts to $\norm{U_k}_{\Lambda^{s,p}}=1$ for all $k$.

    \medskip\noindent\textit{Step 1: $|\phi U_k|_{\Lambda^{s,p}}\to 0$.}
    In $\mathscr D'(\R^n)$, write
    \begin{equation*}
        \flap(\phi U_k)=\phi\,\flap U_k+U_k\,\flap\phi+R_k ,
    \end{equation*}
    and show that each term on the right-hand side tends to $0$ in
    $L^p(\R^n)$.
    For the first term, $\|\phi\flap U_k\|_{L^p}\leq\|\phi\|_{L^\infty}\|\flap U_k\|_{L^p}\to 0$. For the second, since $|\flap\phi(x)|\leq C\rho(x)^{-n-s}$ and $\sup_k\|U_k\rho^{-s}\lg^{-1}\|_{L^p}<\infty$, for any $\varepsilon>0$ we may choose $\tau$ large so that
    $$\sup_k\left\|\frac{U_k}{\rho^{n+s}}\right\|_{L^p(B_\tau^\mathrm{c})}<\varepsilon.$$
    Moreover, by the compact embedding $\Lambda^{s,p}(\R^n)\hookrightarrow\hookrightarrow L^p(B_\tau)$ (which follows from $\Lambda^{s,p}\subset W^{s,p}(B_\tau)$ and the Rellich theorem \cite{DPV12}),
    $$\|U_k\flap\phi\|_{L^p(B_\tau)}\leq C\|U_k\|_{L^p(B_\tau)}\to0, \quad\text{as }k\to\infty.$$
    Hence $\|U_k\flap\phi\|_{L^p}\to 0.$

    For the remainder
    $R_k:=\flap(\phi U_k)-\phi\flap U_k-U_k\flap\phi\in\mathscr D'(\R^n)$
    we distinguish two cases. All estimates below are first established
    for $u\in\mathscr S(\R^n)$ and extend to $U_k\in Y$ by density.
    
    If $s<1$, Lemma~\ref{remainder_estimate} applied to the pair
    $(U_k,\phi)$ gives
    \[\norm{R_k}_{L^p}\le c\,\norm{\phi}_{L^\infty}
    \norm{\flap U_k}_{L^p}\to0.\]

    If $1\le s<2$, fix $\eta\in C_{\mathrm{c}}^\infty(B_8)$ with $0\le\eta\le1$ and
    $\eta=1$ on $B_4$, and split $U_k=\eta U_k+w_k$. Since
    $\phi\,w_k\equiv0$,
    \begin{equation}\label{Rk_split}
        \begin{aligned}
          R_k=&\Big(\flap(\phi\,\eta U_k)-\phi\,\flap(\eta U_k)
          -(\eta U_k)\,\flap\phi\Big)\\
          &\;-\;\phi\,\flap w_k\;-\;w_k\,\flap\phi.
        \end{aligned}
    \end{equation}
    For the bracket, note $\eta U_k\in W^{1,p}(\R^n)$ with
    $\norm{D(\eta U_k)}_{L^p}\le C\norm{U_k}_{\Lambda^{s,p}}=C$, the
    weights of Definition~\ref{fS} being comparable to constants on
    $B_8$; hence Lemma~\ref{remainder_estimate}, combined with \eqref{GN_first} at
    the order $s-1$ when $1<s<2$, yields
    \[
      \Norm{\flap(\phi\,\eta U_k)-\phi\,\flap(\eta U_k)
      -(\eta U_k)\,\flap\phi}_{L^p}
      \lesssim_{\phi}\norm{\eta U_k}_{L^p}^{2-s}\,\norm{D(\eta U_k)}_{L^p}^{s-1},
    \]
    and the bound tends to $0$, since
    $\norm{\eta U_k}_{L^p}\le\norm{U_k}_{L^p(B_8)}\to0$ by the compact
    embedding used above. For the second term in \eqref{Rk_split}: $w_k$
    vanishes on $B_4$, so for
    $x\in\operatorname{supp}\phi\subset\bar B_2$ the distribution
    $\flap w_k$ is given by the absolutely convergent integral
    $-c_{n,s}\int_{B_4^{\,c}}w_k(y)|x-y|^{-n-s}\,dy$, where
    $|x-y|\ge|y|/2$; hence, for every $\tau>4$,
    \[
      \norm{\phi\,\flap w_k}_{L^p}
      \lesssim_{\phi}\int_{B_4^{\,c}}
      \frac{|U_k(y)|}{\rho(y)^{n+s}}\,dy
      \le C_\tau\norm{U_k}_{L^p(B_\tau)}
      +\Norm{\frac{U_k}{\rho^{s}\lg}}_{L^p(\R^n)}
       \Norm{\frac{\lg}{\rho^{n}}}_{L^{p'}(B_\tau^{\,c})} ,
    \]
    sending $k\to\infty$ (compact embedding) then
    $\tau\to\infty$ shows $\norm{\phi\,\flap w_k}_{L^p}\to0$. Finally,
    $\norm{w_k\,\flap\phi}_{L^p}
    \le\norm{U_k\flap\phi}_{L^p}
    +\norm{\flap\phi}_{L^\infty}\norm{U_k}_{L^p(B_8)}\to0$, the first
    term by the preceding paragraph. Hence $\norm{R_k}_{L^p}\to0$ in all
    cases, and therefore $|\phi U_k|_{\Lambda ^{s,p}}\to0$.

    \medskip\noindent\textit{Step 2: $\|\phi U_k\|_{\Lambda^{s,p}}\to 0$.}
    We upgrade this to convergence of step 1 in the full norm. Since
    $\operatorname{supp}(\phi U_k)\subset\bar B_2$, where the weights of
    Definition~\ref{fS} are comparable to constants,
    \[
      \norm{\phi U_k}_{\Lambda^{s,p}}\lesssim
      \norm{\phi U_k}_{L^p}+\norm{\flap(\phi U_k)}_{L^p},
    \]
    with the additional term $\norm{D(\phi U_k)}_{L^p}$ on the right when
    $1\le s<2$. The first term is at most
    $\norm{\phi}_{L^\infty}\norm{U_k}_{L^p(B_2)}\to0$ by the compact
    embedding, and the second tends to $0$ as just shown. The gradient term tends to $0$ due to the interpolation \eqref{GN_grad} and density.

    \medskip\noindent\textit{Step 3: $\|\psi U_k\|_{\Lambda^{s,p}}\to 0$.} Since $\psi$ is supported in $\R^n\setminus B_1$ and $U_k\in Y$, the function $\psi U_k$ can be approximated in $\Lambda^{s,p}$ by $C_{\mathrm{c}}^\infty$ functions supported in $B_1^c$. The critical Hardy inequality \eqref{critical_Hardy_inequality} (for $s<1$), Corollary~\ref{Poincare_sp_n} (for $s\ge 1, sp=n$) or the last case of Corollary~\ref{Poincare_1_s_2_noncritical} (for $s\ge1, (s-1)p=n$) then gives
    \begin{align*}
        \norm{\psi U_k}_{\Lambda^{s,p}}\lesssim|\psi U_k|_{\Lambda^{s,p}}\leq|U_k|_{\Lambda^{s,p}}+|\phi U_k|_{\Lambda^{s,p}}\to 0.
    \end{align*}

    Hence $U_k\to 0$ strongly in $\Lambda^{s,p}$, contradicting $\norm{U_k}_{\Lambda^{s,p}}=1$.
\end{proof}

\subsection{Isomorphism on $Y$}

We first recall the kernel description:
\begin{theorem}[Fall \cite{Fall16}]\label{Fall}
   Suppose $s\in (0,2)$, every $s$-harmonic function ($\flap u = 0$ in $\R^n$ in the sense of distribution) in $\R^n$ is affine, and constant if $s\in (0,1]$.
\end{theorem}

\begin{lemma}\label{poly_in_Y}
    $\mathcal{P}_{[s-n/p]}\subset Y$.
\end{lemma}
\begin{proof}
It suffices to consider $sp\ge n$.
    Choose $\eta\in C_{\mathrm{c}}^\infty(B_2)$ with $\eta=1$ on $B_1$ and set $\eta_R(x)=\eta(x/R)$. A direct computation yields
    $$\Big\|\frac{\eta_R-1}{\rho^s}\Big\|_{L^p}\to0~~(sp>n),\qquad\Big\|\frac{\eta_R-1}{\rho^s\lg}\Big\|_{L^p}\to0~~(sp=n),$$
    $$\Big\|\frac{D\eta_R}{\rho^{s-1}}\Big\|_{L^p}\lesssim R^{n/p-s},\qquad\|\flap\eta_R-\flap1\|_{L^p}=R^{n/p-s}\|\flap\eta\|_{L^p}.$$
    If $sp>n$, then $\|\eta_R-1\|_{\Lambda^{s,p}}\to0$, so $1\in Y$.

    If $sp=n$, the above shows that $\{\eta_R\}_{R>1}$ is bounded in
    $\Lambda^{s,p}$. Since $\Lambda^{s,p}$ is reflexive, there is a
    subsequence $\{\eta_{R_k}\}$, $R_k\to\infty$, converging weakly to some
    $\tilde u\in\Lambda^{s,p}$. For every
    $\varphi\in C_{\mathrm{c}}^\infty(\R^n)$, the functional
    $u\mapsto\int_{\R^n}u\varphi$ is continuous on $\Lambda^{s,p}$, since
    $\big|\int u\varphi\big|
    \le\norm{u\rho^{-s}\lg^{-1}}_{L^p}\norm{\varphi\,\rho^{s}\lg}_{L^{p'}}$
    and the second factor is finite by the compactness of
    $\mathrm{supp}\,\varphi$. As $\eta_{R_k}=1$ on $\mathrm{supp}\,\varphi$
    for all large $k$, we obtain $\int\tilde u\varphi=\int\varphi$, hence
    $\tilde u=1$ a.e. Mazur's theorem then yields convex combinations of the
    $\eta_{R_k}\in C_{\mathrm{c}}^\infty(\R^n)$ converging strongly to $1$,
    so $1\in Y$.

    Finally, let $s-\frac np\ge1$, so that $s>1$ and
    $\mathcal P_{[s-n/p]}=\mathcal P_1$; since $Y$ is a linear subspace and
    $1\in Y$, it suffices to show $x_j\in Y$ for each $j$. Note first that
    $x_j\in\cL_s$ (as $s>1$) and $\flap x_j=0$, the symmetric difference of
    an affine function vanishing identically. Writing
    $\eta_R(x)x_j=R\,w(x/R)$ with $w(x):=\eta(x)x_j\in
    C_{\mathrm{c}}^\infty(\R^n)$, the scaling of $\flap$ gives
    \[
      \norm{\flap(\eta_Rx_j-x_j)}_{L^p}
      =\norm{\flap(\eta_Rx_j)}_{L^p}
      =R^{\,1-s+\frac np}\norm{\flap w}_{L^p}.
    \]
    For the lower-order terms, $|(\eta_R-1)x_j|\le\rho$ and
    $|D(\eta_Rx_j)-e_j|=|(\eta_R-1)e_j+x_jD\eta_R|\lesssim1$, both supported
    in $\{|x|\ge R\}$. If $s-\frac np>1$, then $(s-1)p>n$ and dominated
    convergence yields
    \[
      \Norm{\frac{(\eta_R-1)x_j}{\rho^{s}}}_{L^p}
      +\Norm{\frac{D(\eta_Rx_j)-e_j}{\rho^{s-1}}}_{L^p}\longrightarrow0;
    \]
    since moreover $1-s+\frac np<0$, we conclude
    $\norm{\eta_Rx_j-x_j}_{\Lambda^{s,p}}\to0$, so $x_j\in Y$. If
    $s-\frac np=1$, then $(s-1)p=n$ and Definition~\ref{fS} places the
    weights $\rho^{-s}\lg^{-1}$ and $\rho^{1-s}\lg^{-1}$ on the orders $0$
    and $1$; as $\rho^{-n}\lg^{-p}\in L^1(\R^n)$ (the same integrability
    showing that $x_j$ itself satisfies these weighted conditions),
    dominated convergence still gives the convergence of both lower-order
    terms, while $\norm{\flap(\eta_Rx_j)}_{L^p}=\norm{\flap w}_{L^p}$
    remains merely bounded. Hence $\{\eta_Rx_j\}_{R>1}$ is bounded in
    $\Lambda^{s,p}$, and the weak-compactness argument of the case $sp=n$
    (testing against $\varphi\in C_{\mathrm{c}}^\infty(\R^n)$ and Mazur's
    theorem) again yields $x_j\in Y$.
\end{proof}

\begin{theorem}\label{iso_on_Y}
    Let $s\in(0,2)$. Then
    \begin{equation}
        \flap :Y/\mathcal{P}_{[s-n/p]}\to L^p(\R^n)
    \end{equation}
    is an isomorphism.
\end{theorem}

To prove the theorem, we introduce the \emph{Lizorkin test space}
\begin{equation}\label{Lizorkin_test_space}
    \begin{aligned}
      \mathscr S_0(\R^n):=&\Big\{\phi\in\mathscr S(\R^n):\
      \int_{\R^n}x^\alpha\phi(x)\,dx=0\ \ \forall\,\alpha\in\mathbb N_0^n\Big\}\\
      =&\Big\{\phi\in\mathscr S:\partial^\alpha\hat\phi(0)=0\ \ \forall\alpha\Big\}.
    \end{aligned}
\end{equation}
This is a subspace of the Schwartz space, its annihilator in $\mathscr S'$ is $\mathscr S_0^\perp=\mathcal P$ (the space of all polynomials), since if $v\in\mathscr S_0^\perp$, then
\[\langle\mathcal Fv,\,w\rangle=\langle v,\,\mathcal Fw\rangle=0,~\forall w\in C^\infty_{\mathrm{c}}(\R^n\setminus\{0\})\implies\operatorname{supp}\,\mathcal{F}v\subset\{0\}.\]
Thus $\mathcal Fv$ is a finite linear combination of derivatives of the Dirac mass, whose (inverse)
Fourier transform is a polynomial (cf.\ \cite[Corollary~2.4.2]{Grafakos}). 

Hence, $\mathscr S_0$ is dense in $L^p$ for $1<p<\infty$. Suppose it fails. By the Hahn--Banach theorem and the Riesz representation, there exists $v\in L^{p'}(\R^n)\setminus\{0\}$ vanishing $\mathscr S_0$, hence $v\in\mathcal P\cap L^{p'}=\{0\}$, a contradiction.

It is remarkable that the operator $\flap:\mathscr S_0\to\mathscr S_0$ is a continuous bijection (the multiplier $|\xi|^s$ preserves the infinite-order flatness of $\hat\phi$ at $\xi=0$.
\renewcommand{\proofname}{Proof of Theorem~\ref{iso_on_Y}}
\begin{proof}
    The fractional Laplacian $\flap:Y\to L^p(\R^n)$ is clearly linear and continuous. By Theorem~\ref{Fall}, the kernel of $\flap$ restricted to $\Lambda^{s,p}$ is contained in $\mathcal{P}_{[s-n/p]}$. Then lemma~\ref{poly_in_Y} asserts $\ker\,\flap|_Y=\mathcal{P}_{[s-n/p]}$.

    We show the image is a closed dense subset in $L^p(\R^n)$, hence the whole $L^p(\R^n)$. Combined with the Poincar\'e inequality (\eqref{sw} for $sp<n$ and Proposition~\ref{Poincare_on_Y} for $sp\ge n$), the map $\flap:Y/\mathcal{P}_{[s-n/p]}\to L^p(\R^n)$ is a bounded injection whose image is closed. Since $\mathscr S_0\subset\Lambda^{s,p}(\R^n)$ and $\flap$ is a bijection on $\mathscr S_0$, the image contains $\mathscr S_0$, a dense subset in $L^p(\R^n)$.
\end{proof}
\renewcommand{\proofname}{Proof}

\subsection{$Y=\Lambda^{s,p}(\R^n)$ and density of $C_{\mathrm{c}}^\infty$}

\begin{theorem}\label{Y_equals_Lambda}
    $Y=\Lambda^{s,p}(\R^n)$. In particular, $C_{\mathrm{c}}^\infty(\R^n)$ is dense in $\Lambda^{s,p}(\R^n)$, and
    $$\flap:\Lambda^{s,p}(\R^n)/\mathcal{P}_{[s-n/p]}\to L^p(\R^n)$$
    is an isomorphism.
\end{theorem}
\begin{proof}
    Let $u\in\Lambda^{s,p}(\R^n)$. Then $\flap u\in L^p(\R^n)$. By the surjectivity in Theorem~\ref{iso_on_Y}, there exists $v\in Y$ such that $\flap v=\flap u$. Therefore $u-v$ is $s$-harmonic in $\R^n$. Since $u-v\in\Lambda^{s,p}(\R^n)\subset\cL_s$, Theorem~\ref{Fall} implies $u-v\in\mathcal{P}_{[s-n/p]}$. By Lemma~\ref{poly_in_Y}, $\mathcal{P}_{[s-n/p]}\subset Y$, hence $u=v+(u-v)\in Y$.
\end{proof}

The following isomorphism is derived by a duality argument.

\begin{definition}\label{def_dual_space}
For $s\in(0,2)$ and $p\in(1,\infty)$, define the negative-order space
\[
  \Lambda^{-s,p}(\R^n):=\big(\Lambda^{s,p'}(\R^n)\big)'.
\]
For $T\in\Lambda^{-s,p}$ and $q\in\mathcal P_{[s-n/p']}$, we write
$T\perp q$ if $\langle T,q\rangle=0$.
\end{definition}

\begin{theorem}[Dual isomorphism]\label{dual_isom}
Let $s\in(0,2)$, $p\in(1,\infty)$. Then
\[
  \flap:L^p(\R^n)\longrightarrow
  \Lambda^{-s,p}(\R^n)\perp\mathcal P_{[s-n/p']}
\]
is an isomorphism, in the sense that for every
$T\in\Lambda^{-s,p}(\R^n)$ with $T\perp\mathcal P_{[s-n/p']}$, there
exists a unique $h\in L^p(\R^n)$ such that
\[
  \langle T,v\rangle=\int_{\R^n}h\,\flap v\,dx,
  \qquad\forall\,v\in\Lambda^{s,p'}(\R^n),
\]
and $\|h\|_{L^p}\le C\|T\|_{\Lambda^{-s,p}}$.
\end{theorem}

\begin{proof}
By Theorem~\ref{Y_equals_Lambda} (with $p$ replaced by $p'$), the map
$A:\Lambda^{s,p'}(\R^n)/\mathcal P_{[s-n/p']}\to L^{p'}(\R^n)$,
$Av:=\flap v$, is an isomorphism. Its Banach-space adjoint
\[
  A':L^p=(L^{p'})'\longrightarrow
  \big(\Lambda^{s,p'}/\mathcal P_{[s-n/p']}\big)'
  \cong\Lambda^{-s,p}\perp\mathcal P_{[s-n/p']}
\]
is therefore also an isomorphism.  By definition,
$A'$ acts as
\[
  \langle A'h,\,v\rangle
  =\langle h,\,\flap v\rangle_{L^p\times L^{p'}}
  =\int_{\R^n}h\,\flap v\,dx,
  \qquad h\in L^p,\;v\in\Lambda^{s,p'}.
\]
For $\varphi\in C_{\mathrm{c}}^\infty(\R^n)\subset\Lambda^{s,p'}(\R^n)$,
this coincides with the distributional definition of $\flap h$
(Definition~\ref{def_fracLap}):
\[
  \langle A'h,\varphi\rangle
  =\int_{\R^n}h\,\flap\varphi\,dx
  =\langle\flap h,\varphi\rangle_{\mathscr D'\times C_{\mathrm{c}}^\infty}.
\]
Since $C_{\mathrm{c}}^\infty$ is dense in $\Lambda^{s,p'}$
(Theorem \ref{Y_equals_Lambda}) and
$v\mapsto\int h\,\flap v$ is continuous on $\Lambda^{s,p'}$
\emph{(}by $|\!\int h\,\flap v|\le\|h\|_{L^p}\|\flap v\|_{L^{p'}}
\le C\|h\|_{L^p}\|v\|_{\Lambda^{s,p'}}$\emph{)},
the functional $A'h$ is the unique continuous extension of
$\flap h\big|_{C_{\mathrm{c}}^\infty}$ to $\Lambda^{s,p'}$.
We therefore write $\flap h:=A'h\in\Lambda^{-s,p}\perp
\mathcal P_{[s-n/p']}$.
\end{proof}
Note that the orthogonality condition $T\perp\mathcal P_{[s-n/p']}$ is
automatically satisfied for $T=\flap h$ with $h\in L^p$, because
$\langle\flap h,q\rangle=\int h\,\flap q=0$ for every
$q\in\mathcal P_{[s-n/p']}\subset\mathcal P_{[s]}=\ker(\flap)$.

\subsection{Compositions of some differential operators and related isomorphisms}

\begin{proposition}\label{D1_bdd}
    Let $-1<s<2,~1<p<\infty$, then the gradient map $D_1:\Lambda^{s,p}(\R^n)\to\Lambda^{s-1,p}(\R^n).$
\end{proposition}
\begin{proof}

    When $1\leq s<2$, this is just the definition. By duality, it extends to $-1<s\leq 0$. 
    
    If $0<s<1$, then for all $u\in\mathscr{S}(\R^n)\subset\Lambda^{s,p}(\R^n),\varphi\in\mathscr{S}(\R^n)\subset\Lambda^{1-s,p'}(\R^n)$,
    $$|\langle D_1u,\varphi\rangle|=|\langle R_1\flap u,(-\triangle)^{\frac{1-s}{2}}\varphi\rangle|\lesssim\norm{\flap u}_{L^p(\R^n)}\norm{(-\triangle)^{\frac{1-s}{2}}\varphi}_{L^{p'}(\R^n)}.$$
    In this case $D_1=(-\triangle)^{\frac{1-s}{2}}R_1\flap$.
\end{proof}

By similar arguments, fractional Laplacian operators are continuous as well. Combing Theorem~\ref{iso_on_Y} and Theorem~\ref{dual_isom}, we derive the following isomorphism.
\begin{proposition}\label{flap_bdd}
    Let $s,t\in(0,2),~1<p<\infty$. Then $(-\triangle)^{\frac{s+t}{2}}:\Lambda^{s,p}(\R^n)\to\Lambda^{-t,p}(\R^n)$ is bounded. More precisely, it is bounded on $\mathscr{S}(\R^n)$ then extends to $\Lambda^{s,p}(\R^n)$, and equals to $(-\triangle)^{\frac{t}{2}}\circ\flap$. Moreover
    \begin{equation}
        (-\triangle)^{\frac{s+t}{2}}:\Lambda^{s,p}(\R^n)/\mathcal{P}_{[s-n/p]}\to\Lambda^{-t,p}(\R^n)\perp\mathcal{P}_{[t-n/p']}
    \end{equation}
    is an isomorphism.
\end{proposition}

We claim that the space $\Lambda^{1,p}(\R^n)$ coincides with the classical weighted Sobolev space $W_0^{1,p}(\R^n)$ as in \cite{AGG94}.
\begin{proposition}\label{Lambda1=W1}
    Let $1<p<\infty,~n\geq 2$. Then $\Lambda^{1,p}(\R^n)=W_0^{1,p}(\R^n)$ with equivalent norm.
\end{proposition}
\begin{proof}
    The Schwartz class $\mathscr{S}(\R^n)$ is dense in both spaces. Then the assertion follows from the fact that
    $$D_ju=R_j(-\triangle)^{\frac{1}{2}}u,~(-\triangle)^{\frac{1}{2}}u=\sum_jR_jD_ju$$
     and that the Riesz transform $R_j:L^p(\R^n)\to L^p(\R^n)$ is bounded.
\end{proof}

We now consider the case when the order of the
target space remains nonnegative. We first record a decay estimate for fractional Laplacian of Schwartz
functions.

\begin{lemma}\label{schwartz_decay}
    Let $\sigma\in(0,2)$ and $u\in\mathscr S(\R^n)$. Then for every
    multi-index $\lambda$,
    \[
      \big|D^\lambda(-\triangle)^{\frac{\sigma}{2}}u(x)\big|
      \le C_{\lambda,u}\,\rho(x)^{-n-\sigma},
      \qquad x\in\R^n.
    \]
    In particular, $(-\triangle)^{\frac{\sigma}{2}}u\in\cL_\tau$ for every
    $\tau>0$.
\end{lemma}

\begin{proof}
    Since $D^\lambda(-\triangle)^{\sigma/2}u
    =(-\triangle)^{\sigma/2}D^\lambda u$ and $D^\lambda u\in\mathscr S(\R^n)$,
    it suffices to treat $\lambda=0$. Boundedness is clear, so let $|x|\ge2$
    and write
    \[
      (-\triangle)^{\frac{\sigma}{2}}u(x)
      =-\frac{c_{n,\sigma}}{2}\int_{\R^n}
      \frac{u(x+z)+u(x-z)-2u(x)}{|z|^{n+\sigma}}\,dz.
    \]
    On $\{|z|\le|x|/2\}$, the second-order Taylor formula and the rapid decay
    of $D^2u$ give
    \[
      \int_{|z|\le|x|/2}\frac{|u(x+z)+u(x-z)-2u(x)|}{|z|^{n+\sigma}}\,dz
      \lesssim|x|^{-n-2}\int_{|z|\le|x|/2}|z|^{2-n-\sigma}\,dz
      \lesssim|x|^{-n-\sigma}.
    \]
    On $\{|z|>|x|/2\}$, the rapid decay of $u$ yields
    \[
      \int_{|z|>|x|/2}\frac{|u(x+z)|+|u(x-z)|+2|u(x)|}{|z|^{n+\sigma}}\,dz
      \lesssim|x|^{-n-\sigma}.
    \]
\end{proof}

Let $s\in(0,2)$, $p\in(1,\infty)$ and $0<t<s$. It can be easily checked that when $(s-t)p\ge n$ the
distributional definition of $(-\triangle)^{\frac{t}{2}}u$
(Definition~\ref{def_fracLap}) is not available for every
$u\in\Lambda^{s,p}(\R^n)$. As in Proposition~\ref{flap_bdd}, we therefore
define the operator on $\mathscr S(\R^n)$ and extend it by continuity; the
existence of the extension is part of the following statement, and its
consistency with Definition~\ref{def_fracLap} in the range $(s-t)p<n$ is
verified in Remark~\ref{consistency_rem}.

\begin{proposition}\label{scale_isom}
    Let $s\in(0,2)$, $p\in(1,\infty)$ and $0<t\le s$. Then the map
    $u\mapsto(-\triangle)^{\frac{t}{2}}u$, initially defined on
    $\mathscr S(\R^n)$, extends uniquely to an isomorphism
    \[
      (-\triangle)^{\frac{t}{2}}:\
      \Lambda^{s,p}(\R^n)/\mathcal P_{[s-n/p]}
      \longrightarrow
      \Lambda^{s-t,p}(\R^n)/\mathcal P_{[s-t-n/p]},
    \]
    and the extension satisfies
    \begin{equation}\label{flap_factorization}
      (-\triangle)^{\frac{s-t}{2}}\circ(-\triangle)^{\frac{t}{2}}=\flap
      \qquad\text{on }\Lambda^{s,p}(\R^n).
    \end{equation}
\end{proposition}

\begin{proof}
    If $t=s$ the statement is Theorem~\ref{main_isom} under the convention
    above, so assume $t<s$ and set $\sigma:=s-t\in(0,2)$.

    Let $u\in\mathscr S(\R^n)$ and $v:=(-\triangle)^{t/2}u$. By
    Lemma~\ref{schwartz_decay}, $|D^\lambda v|\lesssim\rho^{-n-t}$ for every
    $\lambda$; hence, for $0\le|\lambda|\le[\sigma]$,
    \[
      \Norm{\rho^{|\lambda|-\sigma}D^\lambda v}_{L^p(\R^n)}
      \lesssim\Norm{\rho^{|\lambda|-s-n}}_{L^p(\R^n)}<\infty,
    \]
    because $|\lambda|\le[\sigma]<s$; this implies all the weighted
    conditions of Definition~\ref{fS} at order $\sigma$, the logarithmic
    weight only making them weaker. Moreover $v\in\cL_\sigma\cap L^2(\R^n)$
    by Lemma~\ref{schwartz_decay} and Plancherel's theorem, and for every
    $\phi\in C_{\mathrm{c}}^\infty(\R^n)$ Parseval's relation yields
    \[
    \begin{aligned}
      &\int_{\R^n}v\,(-\triangle)^{\frac{\sigma}{2}}\phi\,dx
      =\int_{\R^n}|\xi|^{t}\hat u(\xi)\,|\xi|^{\sigma}\hat\phi(-\xi)\,d\xi\\
      =&\int_{\R^n}|\xi|^{s}\hat u(\xi)\,\hat\phi(-\xi)\,d\xi
      =\int_{\R^n}(\flap u)\,\phi\,dx,
    \end{aligned}
    \]
    all the integrands being in $L^1$. Hence
    $(-\triangle)^{\sigma/2}v=\flap u\in L^p(\R^n)$ in the sense of
    Definition~\ref{def_fracLap}, so that $v\in\Lambda^{\sigma,p}(\R^n)$ with
    \[
      |v|_{\Lambda^{\sigma,p}}
      =\norm{(-\triangle)^{\frac{\sigma}{2}}v}_{L^p}
      =\norm{\flap u}_{L^p}
      =|u|_{\Lambda^{s,p}}.
    \]

    Theorem~\ref{main_isom}, applied at the orders $s$ and $\sigma$, asserts
    that $\flap:\Lambda^{s,p}/\mathcal P_{[s-n/p]}\to L^p$ and
    $(-\triangle)^{\sigma/2}:\Lambda^{\sigma,p}/\mathcal P_{[\sigma-n/p]}\to
    L^p$ are isomorphisms of Banach spaces. Therefore for every
    $u\in\mathscr S(\R^n)$,
    \begin{equation}\label{two_sided}
      \norm{v}_{\Lambda^{\sigma,p}/\mathcal P_{[\sigma-n/p]}}
      \sim\norm{(-\triangle)^{\frac{\sigma}{2}}v}_{L^p}
      =\norm{\flap u}_{L^p}
      \sim\norm{u}_{\Lambda^{s,p}/\mathcal P_{[s-n/p]}}.
    \end{equation}

    Now we define
    Define $T_0:\mathscr S(\R^n)\to
    \Lambda^{\sigma,p}(\R^n)/\mathcal P_{[\sigma-n/p]}$ by
    $T_0u:=v+\mathcal P_{[\sigma-n/p]}$. By \eqref{two_sided} and
    $\norm{u}_{\Lambda^{s,p}/\mathcal P_{[s-n/p]}}\le
    \norm{u}_{\Lambda^{s,p}}$, the map $T_0$ is bounded for the
    $\Lambda^{s,p}$-norm; since $\mathscr S(\R^n)$ is dense in
    $\Lambda^{s,p}(\R^n)$ (Theorem~\ref{Y_equals_Lambda}), it extends
    uniquely to a bounded operator
    $T:\Lambda^{s,p}(\R^n)\to
    \Lambda^{\sigma,p}(\R^n)/\mathcal P_{[\sigma-n/p]}$. Both sides of
    \eqref{two_sided} are continuous on $\Lambda^{s,p}(\R^n)$, so
    \eqref{two_sided} persists for every $u\in\Lambda^{s,p}(\R^n)$. In
    particular, $T$ vanishes on $\mathcal P_{[s-n/p]}$
    , while $Tu=0$ forces
    $u\in\mathcal P_{[s-n/p]}$, this space being closed. Hence $T$ descends
    to an isomorphism from $\Lambda^{s,p}/\mathcal P_{[s-n/p]}$ onto its
    image, which is closed in
    $\Lambda^{\sigma,p}/\mathcal P_{[\sigma-n/p]}$.

    We next verify \eqref{flap_factorization}. Since the operator of
    Theorem~\ref{main_isom} at order $\sigma$ is defined on the quotient,
    $(-\triangle)^{\sigma/2}$ annihilates $\mathcal P_{[\sigma-n/p]}$, and
    $u\mapsto(-\triangle)^{\sigma/2}Tu$ is therefore a well-defined bounded
    map from $\Lambda^{s,p}(\R^n)$ to $L^p(\R^n)$, as is $u\mapsto\flap u$.
    The two coincide on $\mathscr S(\R^n)$, hence on
    $\Lambda^{s,p}(\R^n)$ by density and continuity, which is
    \eqref{flap_factorization}.

   To verify surjectivity, let $f\in\Lambda^{\sigma,p}(\R^n)$, so that
    $(-\triangle)^{\sigma/2}f\in L^p(\R^n)$. By Theorem~\ref{main_isom}
    there exists $u\in\Lambda^{s,p}(\R^n)$ with
    $\flap u=(-\triangle)^{\sigma/2}f$. Choosing a representative
    $w\in\Lambda^{\sigma,p}(\R^n)$ of $Tu$, identity
    \eqref{flap_factorization} gives $(-\triangle)^{\sigma/2}(w-f)=0$ with
    $w-f\in\Lambda^{\sigma,p}(\R^n)$, then
    $w-f\in\mathcal P_{[\sigma-n/p]}$, that is,
    $Tu=f+\mathcal P_{[\sigma-n/p]}$. 
\end{proof}

\begin{remark}\label{consistency_rem}
    If $(s-t)p<n$, then $\mathcal P_{[s-t-n/p]}=\{0\}$ and the target
    quotient is trivial, so Proposition~\ref{scale_isom} reads: the operator
    \[
      (-\triangle)^{\frac{t}{2}}:\
      \Lambda^{s,p}(\R^n)/\mathcal P_{[s-n/p]}
      \longrightarrow\Lambda^{s-t,p}(\R^n)
    \]
    is an isomorphism. Moreover, in this range the extension coincides with
    the distributional operator of Definition~\ref{def_fracLap}. Indeed,
    $\Lambda^{s,p}(\R^n)\hookrightarrow\cL_t$  if $u_k\in\mathscr S(\R^n)$ converge to
    $u$ in $\Lambda^{s,p}(\R^n)$, then on the one hand
    $(-\triangle)^{t/2}u_k\to Tu$ in $\Lambda^{s-t,p}(\R^n)$, and on the
    other hand, for every $\phi\in C_{\mathrm{c}}^\infty(\R^n)$,
    Lemma~\ref{schwartz_decay} gives
    \[
      \Abs{\int_{\R^n}(u_k-u)\,(-\triangle)^{\frac{t}{2}}\phi\,dx}
      \lesssim\int_{\R^n}\frac{|u_k-u|}{1+|x|^{n+t}}\,dx
      \longrightarrow0,
    \]
    so that $(-\triangle)^{t/2}u_k\to(-\triangle)^{t/2}u$ in
    $\mathscr D'(\R^n)$ as well; hence $Tu=(-\triangle)^{t/2}u$.
\end{remark}

\begin{remark}\label{target_quotient_rem}
    When $(s-t)p\ge n$, the quotient on the target space cannot be removed:
    there is no bounded operator
    $S:\Lambda^{s,p}(\R^n)/\mathcal P_{[s-n/p]}\to\Lambda^{s-t,p}(\R^n)$
    which selects a representative of $(-\triangle)^{t/2}u$ and is onto.
    Indeed, $(s-t)p\ge n$ implies $1\in\Lambda^{s-t,p}(\R^n)$
    (Lemma~\ref{poly_in_Y}); if $S$ were onto, there would exist $u_*$ with
    $Su_*=1$, and \eqref{flap_factorization} would give
    $\flap u_*=(-\triangle)^{(s-t)/2}\,1=0$, hence
    $u_*\in\mathcal P_{[s-n/p]}$ and $Su_*=0$, a contradiction. 
\end{remark}

Specializing to $s=1$ and invoking Proposition~\ref{Lambda1=W1}, we obtain
fractional-order solvability with source space equal to the weighted Sobolev
space of \cite{AGG94}.

\begin{corollary}\label{W1_scale}
    Let $n\ge2$, $p\in(1,\infty)$ and $t\in(0,1]$. Then
    \[
      (-\triangle)^{\frac{t}{2}}:\
      \Lambda^{1,p}(\R^n)/\mathcal P_{[1-n/p]}
      \longrightarrow
      \Lambda^{1-t,p}(\R^n)/\mathcal P_{[1-t-n/p]}
    \]
    is an isomorphism.
\end{corollary}

\begin{proof}
    Combine Proposition~\ref{scale_isom} with $s=1$ and
    Proposition~\ref{Lambda1=W1}.
\end{proof}

When the target order is negative, the quotient is replaced by an
orthogonality constraint.

\begin{corollary}\label{dual_scale_isom}
    Let $0<t\le\sigma<2$ and $p\in(1,\infty)$. For
    $T\in\Lambda^{t-\sigma,p}(\R^n)$ with
    $T\perp\mathcal P_{[\sigma-t-n/p']}$, define
    \[
      \big\langle(-\triangle)^{\tfrac{t}{2}}T,\,w\big\rangle
      :=\big\langle T,\,(-\triangle)^{\tfrac{t}{2}}w\big\rangle,
      \qquad w\in\Lambda^{\sigma,p'}(\R^n),
    \]
    the right-hand side being unambiguous because
    $(-\triangle)^{\frac{t}{2}}w$ is defined modulo
    $\mathcal P_{[\sigma-t-n/p']}$ (Proposition~\ref{scale_isom} with
    $(s,p)$ replaced by $(\sigma,p')$) and $T$ annihilates this space.
    Then
    \[
      (-\triangle)^{\tfrac{t}{2}}:\
      \Lambda^{t-\sigma,p}(\R^n)\perp\mathcal P_{[\sigma-t-n/p']}
      \longrightarrow
      \Lambda^{-\sigma,p}(\R^n)\perp\mathcal P_{[\sigma-n/p']}
    \]
    is an isomorphism, reducing to Theorem~\ref{dual_isom} when $t=\sigma$,
    and
    \begin{equation}\label{dual_factorization}
      (-\triangle)^{\tfrac{t}{2}}\circ(-\triangle)^{\tfrac{\sigma-t}{2}}
      =(-\triangle)^{\tfrac{\sigma}{2}}
      \qquad\text{on }L^p(\R^n),
    \end{equation}
    where $(-\triangle)^{\frac{\sigma-t}{2}}$ and
    $(-\triangle)^{\frac{\sigma}{2}}$ are the operators of
    Theorem~\ref{dual_isom} at orders $\sigma-t$ and $\sigma$.
\end{corollary}

\begin{proof}
    The operator so defined is the Banach-space adjoint of the isomorphism
    of Proposition~\ref{scale_isom} at order $\sigma$ and exponent $p'$,
    under the isometric identifications
    $(\Lambda^{\sigma',p'}/\mathcal P_k)'\cong
    \Lambda^{-\sigma',p}\perp\mathcal P_k$ used in the proof of
    Theorem~\ref{dual_isom}; being the adjoint of an isomorphism, it is an
    isomorphism. The identity \eqref{dual_factorization} follows by taking
    adjoints in \eqref{flap_factorization} at order $\sigma$ and exponent
    $p'$, the adjoint composition reversing the order of the factors, and
    the adjoint of the quotient operator of Theorem~\ref{main_isom} being
    the operator of Theorem~\ref{dual_isom} by its construction. Finally,
    if $T$ is represented by a function of $\cL_t$, testing against
    $\phi\in C_{\mathrm{c}}^\infty(\R^n)$ shows that the extension is consistent with
    Definition~\ref{def_fracLap}.
\end{proof}

\section{Applications to fractional Stokes problem}\label{sec:appstokes}
In this section, we consider the fractional Stokes problem modeling
steady-state viscous fluid flows with nonlocal dissipation in $\R^n$ with
$n\ge2$. The system belongs to the family of the generalized Navier--Stokes
equations, in which the dissipation $-\triangle$ is replaced by
$(-\triangle)^{\alpha}$, $\alpha>0$, that is, $s=2\alpha$ in our notation.
For the associated Cauchy problem, global Leray--Hopf weak solutions exist
for every $\alpha>0$ \cite{Lions69}, global regularity holds for
$\alpha\ge\frac12+\frac n4$ \cite{Lions69,Wu03}, and for
$\alpha>\frac12$ global existence and uniqueness hold for initial data small
in the scaling-critical Besov spaces \cite{Wu06}. For the stationary problem in the whole three-dimensional space, existence
and nonexistence of solutions in Lebesgue and Lorentz settings were obtained
in \cite{JV24}. The fractional Stokes system is the linearization of these
stationary equations about the rest state: the velocity $u$ and the pressure
$P$ satisfy
\begin{equation}\label{fstokes}
    \begin{cases}
        &\flap u +\nabla P = f\\
        & \mathrm{div}\,u = g
    \end{cases} \quad \text{in }\R^n,
\end{equation}
where we further allow an inhomogeneous divergence $g$. Throughout, $u$ and
$f$ are $\R^n$-valued, while $P$ and $g$ are scalar; the equations are
understood in the sense of distributions, as made precise in
Definition~\ref{def_T} below. Existence and uniqueness are then derived by
combining the isomorphism of Theorem~\ref{main_isom} for the velocity with
the weighted theory of \cite{AGG94} for the pressure.

\begin{definition}\label{def_T}
Let $s\in(0,2)$. For $(u,P)\in(\cL_s)^n\times\mathscr S'(\R^n)$ we define
\[
  T(u,P):=\big(\flap u+\nabla P,\ -\mathrm{div}\,u\big)\in
  \mathscr D'(\R^n)^n\times\mathscr D'(\R^n)
\]
by its action on test fields: for all $\phi\in C_{\mathrm{c}}^\infty(\R^n)^n$ and
$\psi\in C_{\mathrm{c}}^\infty(\R^n)$,
\[
  \langle \flap u+\nabla P,\ \phi\rangle
  :=\int_{\R^n} u\cdot\flap\phi\,dx-\langle P,\ \mathrm{div}\,\phi\rangle,
  \qquad
  \langle -\mathrm{div}\,u,\ \psi\rangle:=\int_{\R^n} u\cdot\nabla\psi\,dx .
\]
\end{definition}

Every $u\in(\cL_s)^n$ defines a tempered distribution, and one can check that $\flap u\in\mathscr S'(\R^n)^n$ and the
map $T$ is well defined and continuous. Thus there is no confusions if the test functions are in $\mathscr S(\R^n)^n$.

\begin{lemma}\label{flap_tempered}
    Let $s\in(0,2)$ and $u\in\cL_s$. Then
    \[
      \langle\flap u,\eta\rangle:=\int_{\R^n}u\,\flap\eta\,dx,
      \qquad\eta\in\mathscr S(\R^n),
    \]
    converges absolutely and defines the unique extension of the
    distribution $\flap u$ of Definition~\ref{def_fracLap} to a tempered
    distribution. Consequently, if $F\in\mathscr S'(\R^n)$ and $\flap u=F$
    in $\mathscr D'(\R^n)$, then
    $\langle F,\eta\rangle=\int_{\R^n}u\,\flap\eta\,dx$ for all
    $\eta\in\mathscr S(\R^n)$.
\end{lemma}
\begin{proof}
    By Lemma~\ref{schwartz_decay},
    \[
      \Abs{\int_{\R^n}u\,\flap\eta}
      \le\Big(\sup_{\R^n}\big(1+|x|^{n+s}\big)|\flap\eta|\Big)
        \int_{\R^n}\frac{|u|}{1+|x|^{n+s}}\,dx,
    \]
    and the first factor is bounded by finitely many Schwartz seminorms of
    $\eta$; hence the formula defines a tempered distribution, whose
    restriction to $C^\infty_{\mathrm c}(\R^n)$ is
    Definition~\ref{def_fracLap}. Uniqueness follows from the density of
    $C^\infty_{\mathrm c}(\R^n)$ in $\mathscr S(\R^n)$, and the last
    assertion is then immediate, $F$ and this extension agreeing on
    $C^\infty_{\mathrm c}(\R^n)$.
\end{proof}

The key structural tool is the commutation of $\flap$ with $\mathrm{div}$:

\begin{lemma}\label{commute}
Let $u\in(\cL_s)^n$ with $\mathrm{div}\,u=g\in\cL_s$. Then
$\mathrm{div}(\flap u)=\flap g$ in $\mathscr S'(\R^n)$.
\end{lemma}

\begin{proof}
For $\psi\in\mathscr S(\R^n)$, by lemma~\ref{flap_tempered}
\[\langle\mathrm{div}(\flap u),\psi\rangle
=-\langle\flap u,\nabla\psi\rangle
=-\int u\cdot\flap(\nabla\psi)\,dx.\]
Since $\flap$ and $\nabla$ commute on $\mathscr S(\R^n)$ (both are Fourier
multipliers), $\flap(\nabla\psi)=\nabla(\flap\psi)$. Set $h:=\flap\psi$ and
let $\chi_R=\chi(\cdot/R)$ with $\chi\in C_{\mathrm{c}}^\infty$, $\chi\equiv1$ on $B_1$.
Then $\chi_R h\in C_{\mathrm{c}}^\infty$ and $\mathrm{div}\,u=g$ gives
$\int u\cdot\nabla(\chi_R h)=-\int g\chi_R h$.
The cross-term
\[|\int h\,u\cdot\nabla\chi_R|\lesssim R^{-1}\int_{R\le|x|\le2R}
\frac{|u|}{1+|x|^{n+s}}\to0\]
by the integrability of $\|u\|_{\cL_s}$, and dominated convergence yields
\[\int u\cdot\nabla h=-\int g\,h=\langle\flap g,\psi\rangle.\]
\end{proof}

Now we study the kernel of $T$:
\begin{proposition}\label{kernel}
Let $s\in(0,2)$. Then
\[
  N_s^{\mathrm{St}}:=\ker\big(T\big|_{(\cL_s)^n\times\mathscr S'}\big)
  =\big\{(\lambda,c)\in(\mathcal H_s)^n\times\mathcal P_0:
  \mathrm{div}\,\lambda=0\big\},
\]
where $\mathcal H_s:=\{w\in\cL_s:\flap w=0\}$ equals $\mathcal P_0$
for $0<s\le 1$ and $\mathcal P_1$ for $1<s<2$.
\end{proposition}

\begin{proof}
Let $(u,P)\in(\cL_s)^n\times\mathscr S'$ with $\flap u+\nabla P=0$ and
$\mathrm{div}\,u=0$.
Since $\mathrm{div}\,u=0\in\cL_s$, Lemma~\ref{commute} gives
$\mathrm{div}(\flap u)=\flap 0=0$. Taking $\mathrm{div}$ of the first
equation we obtain $P$ is a harmonic tempered distribution, hence a harmonic polynomial.

Now the first equation reads $\flap u=-\nabla P$ with $\nabla P$ a polynomial. We claim that $u$ is a polynomial. For any $\phi\in\mathscr S_0$, set $\psi:=(\flap)^{-1}\phi\in\mathscr S_0$
(using the bijectivity of $\flap$ on $\mathscr S_0$). Then, 
\[
  \langle u_i,\phi\rangle
  =\langle u_i,\flap\psi\rangle
  =\langle\flap u_i,\psi\rangle
  =\langle-\partial_iP,\psi\rangle
  =0,
\]
where the second and third equalities are Lemma~\ref{flap_tempered}
applied to $u_i$ with $F=-\partial_iP\in\mathscr S'(\R^n)$, and the last
equality holds because $-\partial_iP\in\mathcal P$ and all moments of
$\psi$ vanish.

Since $u_i\in\mathcal P\cap\cL_s$, a degree-$k$ monomial
$|x|^k\in\cL_s$ requires
$\int_1^\infty r^{k+n-1}r^{-n-s}dr<\infty$, i.e.\ $k<s$. Thus
$u_i\in\mathcal P_0$ for $s\le 1$ and $u_i\in\mathcal P_1$ for $1<s<2$,
i.e.\ $u$ is at most affine. But every affine function is $s$-harmonic
(the antisymmetric first-difference vanishes under the principal-value
integral), so $\flap u=0$. The first equation then gives
$\nabla P=-\flap u=0$, hence $P$ is constant.

Therefore
$P\in\mathcal P_0$, $u\in(\mathcal H_s)^n$, $\mathrm{div}\,u=0$. Conversely,
any such pair satisfies $\flap u+\nabla P=0+0=0$ and $\mathrm{div}\,u=0$.
\end{proof}

\subsection{Existence and uniqueness}
	
Taking divergence to the first equation, one can easily check that if $(f,g)\in \mathscr S'\times\cL_s$, and $T(u,P)=(f,g)$ for some $(u,P)\in ((\cL_s)^n,\mathscr S')$, then $(u,P)$ solves the decoupled system
\begin{equation}\label{fstokes'}
	\begin{cases}
	&\lap P=\mathrm{div}\,f-\flap g,\\
	&  \flap u=f-\nabla P
	\end{cases} \quad \text{in }\R^n.
\end{equation}
Though $g\in\cL_s$ is not always assumed, motivated by the above, we study \eqref{fstokes} via \eqref{fstokes'}. We now state and prove the main result of this section.
	
\begin{theorem}\label{stokes_main}
	Let $n\ge2$, $s\in(0,2)$, $p\in(1,\infty)$, and let $f\in L^p(\R^n)^n$ and $g\in V_{s,p}$, i.e.
	\begin{equation}\label{g_condition}
		g\in\begin{cases}
			\Lambda^{s-1,p}(\R^n)\perp\mathcal P_{[1-s-n/p']},~0<s<1;\\
			\Lambda^{s-1,p}(\R^n),~1\leq s<2.
		\end{cases}
	\end{equation}
	Then \eqref{fstokes} has a
	solution $(u,P)\in\Lambda^{s,p}(\R^n)^n\times W_0^{1,p}(\R^n)$. Any two
	solutions in this class differ by an element of
	\begin{equation}\label{stokes_kernel_class}
		\mathcal N_{s,p}
		:=N_s^{\mathrm{St}}\cap\big[\Lambda^{s,p}(\R^n)^n\times W_0^{1,p}(\R^n)\big]
		=\big\{(\lambda,c)\in(\mathcal P_{[s-n/p]})^n\times\mathcal P_{[1-n/p]}:\
		\mathrm{div}\,\lambda=0\big\},
	\end{equation}
	conversely $(u+\lambda,P+c)$ is again a solution for every 
    $(\lambda,c)\in\mathcal N_{s,p}$, and
	\begin{equation}\label{stokes_est}
		\inf_{(\lambda,c)\in\mathcal N_{s,p}}
		\big(\|u-\lambda\|_{\Lambda^{s,p}}+\|P-c\|_{W_0^{1,p}}\big)
		\le C\big(\|f\|_{L^p}+\|g\|_{\Lambda^{s-1,p}}\big).
	\end{equation}
\end{theorem}

To derive the system \eqref{fstokes'} from \eqref{fstokes}, we have to recover
lemma~\ref{commute} for $g$ as in \eqref{g_condition}. Recall $u\in\Lambda^{s,p}(\R^n)\subset\cL_s(\R^n)$ and $\flap g\in\Lambda^{-1,p}(\R^n)\perp\mathcal{P}_{[1-n/p']}\hookrightarrow\mathscr S'(\R^n)$.
\begin{lemma}\label{commute'}
    Let $u\in(\cL_s)^n$ with $\mathrm{div}\,u=g$ for g as in \eqref{g_condition}. Then
    $\mathrm{div}(\flap u)=\flap g$ in $\mathscr S'(\R^n)$.
\end{lemma}
\begin{proof}
    If $1\leq s<2$, H\"older's inequality implis $g\in\cL_s(\R^n)$, then lemma~\ref{commute}
    applies. We focus on the case $0<s<1$. Obviously $\mathrm{div}(\flap u)=\flap g$ holds in
    sense of $\mathscr S_0'(\R^n)$, hence $w:=\mathrm{div}(\flap u)-\flap g$ is a polynomial,
    with degree $d$. If $d\geq0$, let $w_d\neq0$ be its homogeneous part of degree $d$, $\psi\in C^\infty_{\mathrm{c}}(B_2)$
    and $\psi_R:=\psi(\cdot/R),R\geq 2$. Testing $w:=\mathrm{div}(\flap u)-\flap g$ against $\psi_R$,
    by homogeneity,
    \[R^{n+d}\int_{\R^n}w_d\psi+O(R^{n+d-1})=\int_{\R^n}u\,\flap\nabla\psi_R-\langle g,\,\flap\psi_R\rangle.\]
    Here $|\flap\nabla\psi(x)|\leq\rho(x)^{-n-s}$, hence
    \begin{equation*}
        \begin{aligned}
            &\left|\int_{\R^n}u\,\flap\nabla\psi_R\right|\lesssim R^{-1-s}\int_{\R^n}\frac{|u(x)|\,dx}{\rho(x/R)^{n+s}}\\
            \lesssim& R^{-1-s}R^{n+s}\int_{B_R}\frac{|u(x)|dx}{\rho(x)^{n+s}}+R^{-1-s}R^{n+s}\int_{\R^n\setminus B_R}\frac{|u(x)|dx}{|x|^{n+s}}\lesssim R^{n-1}.
        \end{aligned}
    \end{equation*}
    Meanwhile, $|\langle g,\,\flap\psi_R\rangle|\leq\|g\|_{\Lambda^{s-1,p}}[\flap\psi_R]_{\dot\Lambda^{1-s,p'}}=O(R^{n/p'-1})$. Thus, dividing $R^{n+s}$ on both side and sending $R\to\infty$, we derive $\int_{\R^n}w_d\psi=0$ for any $\psi\in C^\infty_{\mathrm{c}}(B_2)$, hence $w_d=0$, a contradiction. Therefore, $w=0$ and $\mathrm{div}(\flap u)=\flap g$ in $\mathscr S'(\R^n)$.
\end{proof}

\begin{remark}\label{g_condition_rem}
	Condition \eqref{g_condition} is the compatibility imposed by the
	pressure equation, and it is necessary for solvability in the class of
	Theorem~\ref{stokes_main}: if
	$(u,P)\in\Lambda^{s,p}(\R^n)^n\times W_0^{1,p}(\R^n)$ solves
	\eqref{fstokes}, then Lemma~\ref{commute'} gives
	$\flap g=\mathrm{div}(\flap u)=\mathrm{div}\,f-\triangle P$, and the
	right-hand side extends to an element of $W_0^{-1,p}(\R^n)$
	annihilating $\mathcal P_{[1-n/p']}$, since
	$\langle\mathrm{div}\,f,\phi\rangle=-\int f\cdot D\phi$ and
	$\langle\triangle P,\phi\rangle=-\int DP\cdot D\phi$ for
	$\phi\in W_0^{1,p'}(\R^n)$, and gradients annihilate constants. The
	orthogonality in \eqref{g_condition} is non-vacuous precisely when
    $p\le n/(n-1)$, in which case
	$\mathcal P_{[1-n/p']}=\mathcal P_0\subset W_0^{1,p'}(\R^n)$ and it
	reads $\langle\flap g,1\rangle=0$.
\end{remark}

\begin{proof}[Proof of Theorem~\ref{stokes_main}]
	The uniqueness is clear by Proposition~\ref{kernel}. Now we prove existence.
	The formula $\psi\mapsto-\int_{\R^n}f\cdot\nabla\psi\,dx$ is bounded on
	$W_0^{1,p'}(\R^n)$ by H\"older's inequality and extends
	$\mathrm{div}\big|_{C_{\mathrm{c}}^\infty}$; hence
	$\mathrm{div}\,f\in W_0^{-1,p}(\R^n)$, and since the formula vanishes on
	constants, $\mathrm{div}\,f\perp\mathcal P_{[1-n/p']}$ (recall
	$[1-n/p']\le0$, as $n\ge2$). Together with \eqref{g_condition} and Proposition~\ref{flap_bdd}
	for $1\leq s<2$ or Corollary~\ref{dual_scale_isom} for $0<s<1$,
	\[h:=\mathrm{div}\,f-\flap g\in W_0^{-1,p}(\R^n)\perp\mathcal P_{[1-n/p']}.\]
	By the isomorphism \cite[(1.7)]{AGG94}
	\[
	\lap:\;W_0^{1,p}(\R^n)/\mathcal P_{[1-n/p]}
	\longrightarrow W_0^{-1,p}(\R^n)\perp\mathcal P_{[1-n/p']},
	\]
	there exists $P\in W_0^{1,p}(\R^n)$ with $\lap P=h$ and
	\begin{equation}\label{P_est}
    	\|P\|_{W_0^{1,p}/\mathcal P_{[1-n/p]}}
		\lesssim\|f\|_{L^p}+\|\flap g\|_{W_0^{-1,p}};
	\end{equation}
	in particular $\nabla P\in L^p(\R^n)^n$ with
	$\|\nabla P\|_{L^p}\le\|P\|_{W_0^{1,p}/\mathcal P_{[1-n/p]}}$.

	By Theorem~\ref{main_isom} there exists $u\in\Lambda^{s,p}(\R^n)^n$ with
	$\flap u=f-\nabla P$ and
	\begin{equation}\label{u_est}
		\|u\|_{\Lambda^{s,p}/\mathcal P_{[s-n/p]}}
		\lesssim\|f-\nabla P\|_{L^p}
		\lesssim\|f\|_{L^p}+\|\flap g\|_{W_0^{-1,p}}.
	\end{equation}
		
	By construction $\lap P=\mathrm{div}\,f-\flap g$ and
	$\flap u=f-\nabla P$, i.e.\ $(u,P)$ solves the decoupled system
	\eqref{fstokes'}; moreover $u\in\Lambda^{s,p}(\R^n)^n,~P\in W_0^{1,p}(\R^n)\subset\mathscr S'(\R^n)$, and
	$g\in \Lambda^{s-1,p}(\R^n)$. Set $w:=\mathrm{div}\,u-g$. Let $\phi\in\mathscr S_0(\R^n)$ and
	$\psi:=(\flap)^{-1}\phi\in\mathscr S_0(\R^n)$, so that $\flap\psi=\phi$, then
	$$\langle w,\phi\rangle=\langle u,\nabla\flap\psi\rangle-\langle g,\flap\psi\rangle
    =\langle f-\nabla P,\nabla\psi\rangle-\langle\lap P-\mathrm{div}\,f,\psi\rangle=0.$$
	Hence $w$ is a polynomial. 
		
	Suppose $w\neq 0$ and $d=\mathrm{deg}\,w$. 
	We claim $d\leq 0$. Let $w_d\ne0$ be its homogeneous part of degree $d$, $w_r=w-w_d$,
	$\psi\in C_{\mathrm{c}}^\infty(B_2)$ and $\psi_R:=\psi(\cdot/R)$, $R\ge2$.
	Testing $\mathrm{div}\,u=g+w$ against $\psi_R$:
	\begin{equation}
	   -\frac{1}{R}\int_{\R^n}u\cdot(\nabla\psi)(x/R)\,dx
		=\int_{\R^n}g\,\psi_R\,dx+\int_{\R^n}w_d\,\psi_R,dx+\int_{\R^n}w_r\,\psi_R\,dx.
	\end{equation}
	Since $\norm{u\rho^{-s}\lg^{-1}}_{L^p}\lesssim\norm{u}_{\Lambda^{s,p}}$ in
	every case of Definition~\ref{fS}, we have
	$\|u\|_{L^p(B_{2R})}\lesssim_u R^{\,s}\ln R$, so the left-hand side 
	is $O(R^{\,s-1+n/p'}\ln R)$ by H\"older's inequality.
	We estimate the first term in the right-hand side as follows:
	\begin{equation}\label{g_est}
		\left|\int_{\R^n}g\,\psi_R\,dx\right|\leq
		\begin{cases}
			\|g\rho^{1-s}\lg^{-1}\|_{L^p}\|\phi_R\rho^{s-1}\lg\|_{L^{p'}}=O(R^{\,s-1+n/p'}\ln R),~1<s<2;\\
			\|g\|_{\Lambda^{s-1,p}}[\psi_R]_{\Lambda^{1-s,p'}}=O(R^{\,s-1+n/p'}),~0<s\leq 1.
		\end{cases}
	\end{equation}
	By homogeneity we derive
	\[O(R^{\,s-1+n/p'})=R^{n+d}\int_{\R^n}w_d\psi+O(R^{\,n+d-1}),~\int_{\R^n}w_d\psi=O(R^{-1}+R^{\,s-1-n/p-d}).\]
	If $d>s-1-\dfrac{n}{p}$, it follows that $\langle w_d,\psi\rangle=0$
	for all $\psi$, thus $w_d=0$, a contradiction. Therefore $d<s-1-\dfrac{n}{p}<1$.
		
	If $(s-1)p<n$, we have $d<0$, then $\mathrm{div}\,u=g$. If $(s-1)p\ge n$, $d=0$ is
	possible, in which case we wirte $w(x)=c_0$. Then $[s-n/p]=1$ and
	\[\lambda_0:=\frac{c_0}{n}\,x\in(\mathcal P_1)^n\subset\Lambda^{s,p}(\R^n)^n\]
	satisfies $\flap\lambda_0=0$,
	$\mathrm{div}\,\lambda_0=c_0$; replacing $u$ by $u-\lambda_0$, which
	affects neither $\flap u$ nor \eqref{u_est}, yields $\mathrm{div}\,u=g$.
	In either case $(u,P)$ solves \eqref{fstokes}.
	
	It remains to check \eqref{stokes_est}.
	Choose $q\in(\mathcal P_{[s-n/p]})^n$ and $c\in\mathcal P_{[1-n/p]}$
	attaining the quotient norms in \eqref{u_est} and \eqref{P_est}(the infimum over the finite-dimensional space $(\mathcal P_{[s-n/p]})^n$
	is attained, the map $q'\mapsto\|u-q'\|_{\Lambda^{s,p}}$ being continuous
	and coercive). If
	$[s-n/p]\le0$, then $q$ is constant, $\mathrm{div}\,q=0$, so
	$(q,c)\in\mathcal N_{s,p}$ and \eqref{stokes_est} follows. If $[s-n/p]=1$, write $q=a+Bx$ and set
	$\tilde q:=q-\frac{\operatorname{tr}B}{n}\,x$, so that $\mathrm{div}\,\tilde q=0$ and
    $(\tilde q,c)\in\mathcal N_{s,p}$. Fixing
	$\psi\in C_{\mathrm{c}}^\infty(B_2)$ with $\int\psi=1$ and testing
	$\mathrm{div}(u-q)=g-\operatorname{tr}B$ against $\psi$,
	\begin{equation*}
	    \begin{aligned}
	        &|\operatorname{tr}B|
		=\Abs{\int_{\R^n}g\,\psi\,dx+\int_{\R^n}(u-q)\cdot\nabla\psi\,dx}\\
		\lesssim&\|g\|_{\Lambda^{s-1,p}}+\|u-q\|_{L^p(B_2)}
		\lesssim\|g\|_{\Lambda^{s-1,p}}+\|u-q\|_{\Lambda^{s,p}},
	    \end{aligned}
	\end{equation*}
	whence, with $\kappa:=\norm{x}_{\Lambda^{s,p}}$, recalling by Proposition~\ref{flap_bdd}
	and Corollary~\ref{dual_scale_isom}
	\[
		\|u-\tilde q\|_{\Lambda^{s,p}}
		\le\|u-q\|_{\Lambda^{s,p}}+\tfrac{\kappa}{n}\,|\operatorname{tr}B|
		\lesssim\|f\|_{L^p}+\|g\|_{\Lambda^{s-1,p}}.
	\]
	Together with \eqref{P_est}, this proves \eqref{stokes_est}.
\end{proof}

\subsection{Solution of the fractional Stokes equation with $L^p$ pressure}

In this subsection we solve \eqref{fstokes} with the pressure in
$L^p(\R^n)$ and the velocity at the order $s-1$, for the full range
$0<s<2$. We keep the conventions $\Lambda^{0,p}(\R^n)=L^p(\R^n)$ and
$\mathcal P_k=\{0\}$ for $k<0$. To avoid polynomial errors caused by system \eqref{fstokes'},
we take 0-order transforms, the normalized Riesz transforms, to system \eqref{fstokes}:
\begin{equation}\label{riesz_convention}
  \widehat{R_jh}\,(\xi)=\frac{i\xi_j}{|\xi|}\,\widehat h(\xi),
  \qquad R:=(R_1,\dots,R_n),
\end{equation}
so that $D_j=R_j(-\triangle)^{\frac12}$ on $\mathscr S(\R^n)$ and
$\sum_{j=1}^nR_j^2=-\mathrm{Id}$.

The velocity is sought in
\[
  V_{s,p}=\Lambda^{s-1,p}(\R^n)\perp\mathcal P_{[1-s-n/p']},
\]
the orthogonality constraint being nontrivial only for $0<s<1$; for
$s\ge1$ the space $\mathcal P_{[1-s-n/p']}$ is trivial and
$V_{s,p}=\Lambda^{s-1,p}(\R^n)$. To every $u\in V_{s,p}$ we attach its
\emph{$L^p$-realization} $v\in L^p(\R^n)$: for $1<s<2$,
$v:=(-\triangle)^{\frac{s-1}{2}}u$ is the top-order component of
Definition~\ref{fS}; for $s=1$, $v:=u$; for $0<s<1$, $v$ is the unique
solution of $(-\triangle)^{\frac{1-s}{2}}v=u$ furnished by
Theorem~\ref{dual_isom} at the order $1-s$, whose target is precisely
$V_{s,p}$. By Theorem~\ref{main_isom} at the order $s-1$ when $1<s<2$,
trivially when $s=1$, and by Theorem~\ref{dual_isom} when $0<s<1$, the
map $u\mapsto v$ induces an isomorphism
\begin{equation}\label{realization_iso}
  V_{s,p}\big/\mathcal P_{[s-1-n/p]}
  \;\xrightarrow\;L^p(\R^n).
\end{equation}
The equations \eqref{fstokes} are understood through the identities
\begin{equation}\label{stokes_identities}
  \flap u=(-\triangle)^{\frac12}v,
  \qquad
  \nabla P=(-\triangle)^{\frac12}\big[RP\big],
  \qquad
  \mathrm{div}\,u=(-\triangle)^{\frac{2-s}{2}}
  \Big[\sum_{j=1}^nR_jv_j\Big],
\end{equation}
valid for all $u\in V_{s,p}^n$ with realization $v$ and all
$P\in L^p(\R^n)$: the brackets lie in $L^p(\R^n)$, the outer operators
are those of Theorem~\ref{dual_isom} at the orders $1$ and $2-s$, and
consequently the first equation of \eqref{fstokes} holds in
$\Lambda^{-1,p}(\R^n)^n$ and the second in $\Lambda^{s-2,p}(\R^n)$.
These identities are consistent with the operators defined earlier: the
first is the factorization of Proposition~\ref{flap_bdd} for $1<s<2$,
and the dual factorization \eqref{dual_factorization} (with $\sigma=1$,
$t=s$) for $0<s<1$; the second and third hold on $\mathscr S(\R^n)$ as
Fourier multiplier identities under \eqref{riesz_convention} and extend
by continuity and density---in $P\in L^p(\R^n)$ for the second and, for
the third, in $v\in L^p(\R^n)$ when $0<s\le1$ and in
$u\in\Lambda^{s-1,p}(\R^n)$ (Theorem~\ref{Y_equals_Lambda}) when
$1<s<2$---both sides being continuous by Proposition~\ref{D1_bdd} and
Theorem~\ref{dual_isom}.

\begin{theorem}\label{stokes_Lp_pressure}
Let $n\ge2$, $0<s<2$, $p\in(1,\infty)$, and let
$f\in\Lambda^{-1,p}(\R^n)^n$, $g\in\Lambda^{s-2,p}(\R^n)$. Then
\eqref{fstokes} admits a solution $(u,P)\in V_{s,p}^n\times L^p(\R^n)$ if
and only if
\begin{equation}\label{Lp_pressure_orth}
  f\perp\mathcal P_{[1-n/p']}
  \qquad\text{and}\qquad
  g\perp\mathcal P_{[2-s-n/p']}.
\end{equation}
In this case the pressure $P$ is unique, and the solutions are precisely
the pairs $(u+\lambda,P)$ with
$\lambda\in(\mathcal P_{[s-1-n/p]})^n$; thus the solution is unique
whenever $s-1<n/p$, in particular for all $0<s\le1$, and unique up to an
additive constant vector when $s-1\ge n/p$. Moreover every solution
satisfies
\begin{equation}\label{Lp_pressure_est}
  \norm{P}_{L^p}
  +\inf_{\lambda\in(\mathcal P_{[s-1-n/p]})^n}
   \norm{u-\lambda}_{\Lambda^{s-1,p}}
  \le C\big(\norm{f}_{\Lambda^{-1,p}}+\norm{g}_{\Lambda^{s-2,p}}\big).
\end{equation}
\end{theorem}

\begin{proof}
\emph{Necessity.} Let $(u,P)$ be a solution with realization $v$. By
\eqref{stokes_identities}, each of $\flap u$, $\nabla P$,
$\mathrm{div}\,u$ has the form $(-\triangle)^{\frac{\tau}{2}}h$ with
$h\in L^p(\R^n)$ and $\tau\in\{1,\,2-s\}$, and such distributions
annihilate $\mathcal P_{[\tau-n/p']}$: by the defining formula in
Theorem~\ref{dual_isom},
$\langle(-\triangle)^{\frac{\tau}{2}}h,\,q\rangle
=\int_{\R^n}h\,(-\triangle)^{\frac{\tau}{2}}q\,dx=0$, since every
$q\in\mathcal P_{[\tau-n/p']}$ is at most affine---affine only when
$\tau=2-s>1$---and hence $\tau$-harmonic (cf.\ the proof of
Proposition~\ref{kernel}). Therefore
$f=\flap u+\nabla P\perp\mathcal P_{[1-n/p']}$ and
$g=\mathrm{div}\,u\perp\mathcal P_{[2-s-n/p']}$.

\emph{Existence.} Assume \eqref{Lp_pressure_orth}. By
Theorem~\ref{dual_isom} at the orders $1$ and $2-s$, there exist unique
$w\in L^p(\R^n)^n$ and $Q\in L^p(\R^n)$ with
\[
  (-\triangle)^{\frac12}w=f,\qquad
  (-\triangle)^{\frac{2-s}{2}}Q=g,\qquad
  \norm{w}_{L^p}+\norm{Q}_{L^p}
  \lesssim\norm{f}_{\Lambda^{-1,p}}+\norm{g}_{\Lambda^{s-2,p}}.
\]
Set
\[
  P:=Q-\sum_{i=1}^nR_iw_i\in L^p(\R^n),
  \qquad
  v:=w-RP\in L^p(\R^n)^n,
\]
so that $\norm{P}_{L^p}+\norm{v}_{L^p}\lesssim
\norm{w}_{L^p}+\norm{Q}_{L^p}$, and, by
$\sum_iR_i^2=-\mathrm{Id}$,
\[
  \sum_jR_jv_j
  =\sum_jR_jw_j-\sum_jR_j^2P
  =\sum_jR_jw_j+P=Q.
\]
By \eqref{realization_iso} there exists $u\in V_{s,p}^n$ with realization
$v$ and
\[
  \inf_{\lambda\in(\mathcal P_{[s-1-n/p]})^n}
  \norm{u-\lambda}_{\Lambda^{s-1,p}}
  \lesssim\norm{v}_{L^p}
  \lesssim\norm{f}_{\Lambda^{-1,p}}+\norm{g}_{\Lambda^{s-2,p}}.
\]
Then \eqref{stokes_identities} gives
\[
  \flap u+\nabla P
  =(-\triangle)^{\frac12}(v+RP)
  =(-\triangle)^{\frac12}w=f,
\]
\[
  \mathrm{div}\,u
  =(-\triangle)^{\frac{2-s}{2}}\sum_jR_jv_j
  =(-\triangle)^{\frac{2-s}{2}}Q=g,
\]
so $(u,P)$ solves \eqref{fstokes} and satisfies
\eqref{Lp_pressure_est}.

\emph{Uniqueness.} Let $(u,P)$ solve \eqref{fstokes} with $(f,g)=(0,0)$,
and let $v$ be the realization of $u$. Setting $w_0:=v+RP$ and
$Q_0:=\sum_jR_jv_j$, both in $L^p(\R^n)$, identities
\eqref{stokes_identities} give $(-\triangle)^{\frac12}w_0=0$ and
$(-\triangle)^{\frac{2-s}{2}}Q_0=0$, so the injectivity in
Theorem~\ref{dual_isom} yields $w_0=0$ and $Q_0=0$. Then
\[
  0=Q_0=\sum_jR_j(w_{0,j}-R_jP)
  =\sum_jR_jw_{0,j}+P=P,
\]
whence $P=0$ and $v=w_0-RP=0$; by the injectivity in
\eqref{realization_iso}, $u\in(\mathcal P_{[s-1-n/p]})^n$. Conversely,
when $s-1\ge n/p$ (which forces $1<s<2$), every constant vector
$\lambda$ belongs to $V_{s,p}^n$ (Lemma~\ref{poly_in_Y}) with realization
$0$, so \eqref{stokes_identities} gives $\flap\lambda=0$ and
$\mathrm{div}\,\lambda=0$, and $(u+\lambda,P)$ is again a solution.
Finally, since $P$ is unique and the velocities form exactly the coset
$u+(\mathcal P_{[s-1-n/p]})^n$, the left-hand side of
\eqref{Lp_pressure_est} is an invariant of the solution set, and the
bound proved for the constructed solution holds for every solution.
\end{proof}

{\bf{Acknowledgements.}} 
 
 Li, Ouyang and Wang is partially supported by National Natural Science
Foundation of China (NSFC Grant No. W2531006) and the Institute of Modern Analysis-A Frontier Research Center of Shanghai.
 \medskip

{\bf{Date availability statement:}} Data will be made available on reasonable request.
\medskip

{\bf{Conflict of interest statement:}} There is no conflict of interest.

\end{document}